\documentclass[12pt, a4paper]{amsart}
\pdfoutput=1   

\usepackage[T1]{fontenc}
\usepackage{lmodern}
\usepackage[tracking]{microtype}
\usepackage[utf8]{inputenc}

\usepackage[style=numams, 
            sorting=nyt, 
            isbn=false,  
            backref=true, 
            backend=biber,
            sortcites = true,
            maxbibnames=5,
            maxitems= 1]{biblatex}
\usepackage{amsmath}
\usepackage{amsthm}
\usepackage{amssymb}
\usepackage{mathtools}
\usepackage{dsfont}  

\usepackage{enumitem}

\usepackage{tikz}
\usepackage{tikz-cd}  

\usepackage[colorlinks, 
            linkcolor=blue, 
            citecolor=blue, 
            urlcolor = purple,
           ]{hyperref}

\usepackage[hcentering, hscale=0.71, vscale=0.795 ]{geometry}

\newtheorem{introthm}{Theorem}

\swapnumbers
\newtheorem{thm}{Theorem}[section]
\newtheorem{lemma}[thm]{Lemma}
\newtheorem{prp}[thm]{Proposition}
\newtheorem{co}[thm]{Corollary}

\theoremstyle{definition}
\newtheorem{defn}[thm]{Definition}
\newtheorem{expl}[thm]{Example}

\newtheorem{rmk}[thm]{Remark}
\newtheorem{ques}[thm]{Question}

\newlist{enumthm}{enumerate}{1}  
\setlist[enumthm,1]{label=\textup{(\roman*)}}

\renewcommand{\phi}{\varphi}
\newcommand{\eps}{\varepsilon}
\renewcommand{\rho}{\varrho}

\renewcommand{\geq}{\geqslant}
\renewcommand{\leq}{\leqslant}

\DeclareMathOperator{\conv}{conv}

\DeclareMathOperator{\Sym}{Sym}
\DeclareMathOperator{\GL}{GL}

\DeclareMathOperator{\Tr}{Tr}
\DeclareMathOperator{\Gen}{Gens}
\DeclareMathOperator{\Ker}{Ker}
\DeclareMathOperator{\End}{End}
\DeclareMathOperator{\Irr}{Irr}
\DeclareMathOperator{\Lin}{Lin}
\DeclareMathOperator{\id}{id}
\DeclareMathOperator{\Ann}{Ann}
\DeclareMathOperator{\NKer}{NKer}

\DeclareMathOperator{\Syl}{Syl}

\DeclareMathOperator{\Zent}{\mathbf{Z}}  
\DeclareMathOperator{\Cent}{\mathbf{C}}
\DeclareMathOperator{\ord}{\mathbf{o}}

\DeclareMathOperator{\FS}{\nu_2}       

\newcommand{\kk}{\mathds{k}}
\newcommand{\EE}{\mathds{E}}
\newcommand{\con}[1]{\overline{#1}}

\newcommand{\reals}{\mathds{R}}
\newcommand{\compl}{\mathds{C}}
\newcommand{\QQ}{\mathds{Q}}

\newcommand{\GF}[1]{\mathds{F}_{#1}}

\newcommand{\iso}{\cong}

\DeclarePairedDelimiter{\card}{\lvert}{\rvert}

\DeclarePairedDelimiter{\erz}{\langle}{\rangle}

\tikzset{orbitnodes/.style={circle, draw, fill=black,
                          inner sep=0pt, minimum width=4pt}}

\hypersetup{pdfauthor={Erik Friese and Frieder Ladisch}, 
            pdftitle={Classification of Affine Symmetry Groups
                     of Orbit Polytopes},
            pdfkeywords={Orbit polytopes, group representations, %
                         affine symmetry, linear group, generic symmetry}
            }

\begin{document}
\title[Affine Symmetry Groups of Orbit Polytopes]{Classification of Affine Symmetry Groups\\
of Orbit Polytopes}
\author{Erik Friese}
\email{erik.friese@uni-rostock.de}
\author{Frieder Ladisch}
\thanks{Authors partially supported by the DFG 
(Project: SCHU 1503/6-1).}
\email{frieder.ladisch@uni-rostock.de}
\address{Universität Rostock,
         Institut für Mathematik,         
         18051 Rostock,
         Germany}
%
%
%

\begin{abstract}
    Let $G$ be a finite group
    acting linearly on a vector space $V$.
    We consider the linear symmetry groups
    $\operatorname{GL}(Gv)$ of orbits $Gv\subseteq V$,
    where  
    the \emph{linear symmetry group} $\operatorname{GL}(S)$ 
    of a subset $S\subseteq V$ 
    is defined as the set of all linear maps
    of the linear span of $S$ which permute $S$.
    We assume that $V$ is the linear span of at least
    one orbit $Gv$.
    We define a set of \emph{generic points}
    in $V$, which is Zariski-open in $V$,
    and show that the groups
    $\operatorname{GL}(Gv)$ for $v$ generic are all isomorphic, 
    and isomorphic to a subgroup of every
    symmetry group $\operatorname{GL}(Gw)$ such that 
    $V$ is the linear span of $Gw$.
    If the underlying characteristic is zero,
    ``isomorphic'' can be replaced by
    ``conjugate in $\operatorname{GL}(V)$''.
    Moreover, in the characteristic zero case, we show 
    how the character of $G$ on $V$ determines this
    generic symmetry group.
    We apply our theory to classify all affine symmetry groups
    of vertex-transitive polytopes,
    thereby answering a question of Babai (1977).
%
%
\end{abstract}

\maketitle  

\section{Introduction}

An \emph{orbit polytope} is the convex hull of an orbit of a finite
group acting affinely on a real vector space.
Orbit polytopes can be seen as some kind
of building block for polytopes with symmetries in general,
and turn up in a number of combinatorial
optimization problems and other applications,
and so have been studied by a number of people
\cite{Babai77,barvivershik88,EllisHS06,mccarthyOZZ03,onn93,robertson84pas,sanyaletal11,zobin94}.

A \emph{representation polytope} is the
convex hull of a finite matrix group over the reals.
This is a special case of an orbit polytope,
since a matrix group $G =G\cdot I$ can be interpreted as the orbit
of the identity matrix under left multiplication.
Representation polytopes
have also received considerable attention, especially
permutation polytopes (the convex hull of a finite group
of permutation matrices)
\cite{guralnickperkinson06,BHNP09,McCarthyOSZ02,HofmannNeeb14}.

In a
previous paper \cite{FrieseLadisch16}, 
we developed a general theory
of affine symmetries of orbit polytopes.
(An \emph{affine symmetry} of some point set
$S\subseteq \reals^d$ is a permutation of $S$
that is the restriction of an affine map of the ambient space.
The \emph{affine symmetry group} of $S$ is the group of 
all affine symmetries of $S$.)
The different symmetry groups of polytopes are interesting
since their knowledge can be useful for practical
computations with polytopes~\cite{bdss09,BDPRS14}.
Moreover, the affine symmetries of a finite point set
in $\reals^d$ can be computed effectively~\cite{bdss09}.

One of the main results of the present paper is 
the classification of all finite groups 
which are isomorphic to 
the affine symmetry group of an orbit polytope
(Theorem~\ref{ithm:class_aff} below).
This answers
an old question of Babai~\cite{Babai77}.
We also classify affine symmetry groups of orbit polytopes
with integer coordinates. 

To obtain these results, we develop a general theory
of linear symmetry groups of orbits,
which is a natural continuation of our
previous paper~\cite{FrieseLadisch16}.
The methods are more algebraic than geometric,
and in particular, we first work over an arbitrary field $\kk $.
Thus let $G$ be a finite group acting linearly on a 
vector space $V$ over some field $\kk $.
(Later, we will specialize to $\kk =\reals$, 
the field of real numbers.)
Consider an orbit $Gv$ of a vector $v\in V$.
Its \emph{linear symmetry group} 
$\GL(Gv)$ is the group of all linear automorphisms of 
the linear span of $ Gv$,
which  map the orbit $Gv$ onto itself.
(When $\kk =\reals$, this is also the linear symmetry
group of the corresponding orbit polytope.)
Clearly, any element of $ G$ yields such a symmetry.
It depends on the abstract group $G$,
on the concrete action of $G$ on $V$ and on the choice of $v$, 
whether there are other linear symmetries or not.

To motivate the next definition,
let us look at a (very simple) example.
Let $G= C_4=\erz{t}$, the cyclic group of order $4$,
act on $V=\reals^2$, such that 
$t$ acts as a rotation
by a right angle.
Then the orbit of every point except $v=0$
consists of four points forming a square, 
and the linear symmetry group is always isomorphic
to 
the dihedral group 
$D_4$ of order $8$ (Figure~\ref{fig:rot4}).
Note that the different orbits of $G=\erz{t}$, 
as $v$ varies, do not 
admit exactly \emph{the same} symmetries 
(there are reflections in different lines).
But if we identify the
points of an orbit with the corresponding group elements,
then the linear symmetry groups of all 
orbits induce the same permutations on $G$.

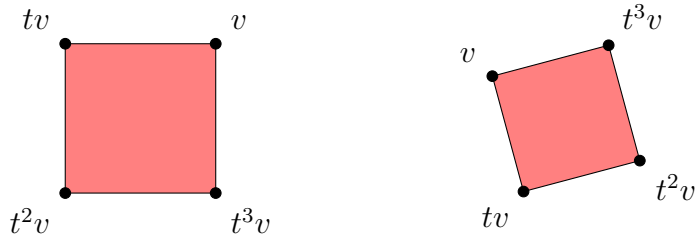
\begin{figure}
  \begin{tikzpicture}[scale=0.7]
    \newcommand{\startangle}{45}
    \newcommand{\radius}{2}
    \filldraw[fill=red!50]
      (\startangle:\radius)  node[label= 
      \startangle:{$v$},orbitnodes]{}
      \foreach \i/\label in {1/t,2/t^2,3/t^3}{
        -- (90*\i + \startangle:\radius) 
            node[label=90*\i + \startangle: $\label v$, orbitnodes] {}
      }
      -- cycle;
    \begin{scope}[xshift=8cm]
        \renewcommand{\startangle}{150}
        \renewcommand{\radius}{1.6}
        \filldraw[fill=red!50]
          (\startangle:\radius)  node[label= 
          \startangle:{$v$},orbitnodes]{}
          \foreach \i/\label in {1/t,2/t^2,3/t^3}{
             -- (90*\i + \startangle:\radius) 
                 node[label=90*\i + \startangle: $\label v$, 
                 orbitnodes] {}
          }
         -- cycle;
    \end{scope}
  \end{tikzpicture}
  \caption{Two orbits of 
             the group $G=\erz{t}$
             of rotations preserving a square.
          }
  \label{fig:rot4}  
\end{figure}

In general, instead of the linear symmetries
of an orbit,
we will consider the permutations of the group $G$
which correspond to linear symmetries.
We call a permutation $\pi \colon G \to G$ an 
\emph{orbit symmetry with respect to $v$} if there is an 
$\alpha \in \GL(V)$ such that 
$\alpha(gv) = \pi(g) v$ for all $g \in G$,
and we write  $\Sym(G,v)$ for the set 
of all orbit symmetries with respect to $v$. 
This is a subgroup of the symmetric group
$\Sym(G)$ on $G$.

Using this definition, we can compare
$\Sym(G,v)$ for different $v$.
We usually consider only $v$ such that
$V = \kk Gv$.
In this case, we call $v$ a \emph{generator} of $V$,
and $V$ a \emph{cyclic} $\kk G$-module.
Define
\[ \Gen(V) :=
   \{ v\in V \colon
      V = \kk G v
   \}\,.
\]
The set of generators $\Gen(V)$ is always Zariski-open
in $V$, and thus is 
``almost all'' of $V$, when $\Gen(V)\neq \emptyset$.
For $V$ such that $\Gen(V)\neq \emptyset$,
we set
\[ \Sym(G,V):= 
    \bigcap_{v\in \Gen(V)}
    \Sym(G,v)\,.
\]

\begin{introthm}\label{ithm:gen_pts_zariski}
Let\/ $\kk $ be a field of infinite order,
$G$ a finite group and $V$
a $\kk G$-module.
The set of $v\in \Gen(V)$
such that $\Sym(G,v) \neq \Sym(G,V)$
is Zariski-closed in $\Gen(V)$. 
\end{introthm}

This means that 
$\Sym(G,V) = \Sym(G,v)$ for ``almost all'' $v \in V$.
We call $\Sym(G,V)$ the generic symmetry group of the 
$\kk G$-module $V$.
Theorem~\ref{ithm:gen_pts_zariski} is a straightforward
generalization of an earlier result~\cite{FrieseLadisch16}
from the case $\kk = \reals$ to arbitrary fields.

If $\kk $ is a field of characteristic zero, 
then the isomorphism type of $V$ as $\kk G$-module
is determined by the character $\chi$ of $G$ on $V$,
and thus $\Sym(G,V)$ is also determined by $\chi$.
The next result shows how to compute
$\Sym(G,V)$ from $\chi$ 
and its decomposition into irreducible characters.
Suppose that $\kk \subseteq \compl$,
the field of complex numbers
(in fact, any algebraically closed field
of characteristic zero would do),
and let $\Irr G$ be the set of irreducible,
complex valued characters 
(or with values in a fixed algebraically 
closed field containing $\kk $).
Then we can write $\chi$ in a unique way as sum
of irreducible characters:
\[ \chi = 
   \sum_{\psi \in \Irr(G)} m_{\psi} \psi \,.
\]
When $V$ is a cyclic $\kk G$-module,
that is, when $V = \kk G v$ for some $v\in V$, then
$m_{\psi}\leq \psi(1)$ for all $\psi\in \Irr(G)$.
For reasons explained later, we call
\[ \chi_I := \sum_{ \substack{ \psi\in \Irr(G) 
                               \\
                               m_{\psi} = \psi(1)
                             }
                  } \psi(1)\psi 
\]
the \emph{ideal part of $\chi$}.
With this notation, we have:

\begin{introthm}\label{ithm:gen_syms_char}
  Let $\chi$ be the character of the cyclic $\kk G$-module~$V$,
  where $\kk $ has characteristic zero,
  and let $N= \Ker(\chi-\chi_I)$.
  Then a permutation $\pi\in \Sym(G)$
  is in $\Sym(G,V)$ if and only if
  the following two conditions hold:
  \begin{enumthm}
  \item $\chi_I \big( \pi(g)^{-1} \pi(h) \big) = \chi_I(g^{-1}h)$
        for all $g$, $h\in G$, and
  \item $\pi(g N) = \pi(1)gN$ (as sets) for all $g\in G$.
  \end{enumthm}  
\end{introthm}

For example, this means that when $N=\{1\}$,
then $\Sym(G,V)$ contains only left multiplications
with elements from $G$,
and thus $\GL(Gv)\iso G$ for ``almost all'' $v\in V$.

When it happens that $\chi = \chi_I$, then the second 
condition in Theorem~\ref{ithm:gen_syms_char} is void.
This special case of Theorem~\ref{ithm:gen_syms_char} 
is already in our previous paper~\cite[Theorem~D]{FrieseLadisch16}.
In this case, we can identify each orbit $Gv$ with the image
of $G$ under the corresponding representation $G\to \GL(V)$.
This case is equivalent to a 
\emph{linear preserver problem},
namely computing the set of linear transformations
of a matrix ring which map a finite matrix group
onto itself.
This problem has already been studied,
especially in the case of finite reflection
groups \cite{limilligan03,LiSpiZobin04,LiTamTsing02}.

For every $v\in \Gen(V)$, we have a 
representation $D_v \colon \Sym(G,V) \to \GL(V)$.
These representations are all similar when
$\kk $ has characteristic zero, and thus 
all have the same character $\widehat{\chi}$.
We also prove a formula for $\widehat{\chi}$
in terms of $\chi$ (Proposition~\ref{prp:gen_character}).

To obtain Theorem~\ref{ithm:gen_syms_char}, 
we have to generalize some results from
our earlier paper~\cite{FrieseLadisch16} 
to more general fields 
(even for $\kk = \reals$, which is our main interest, 
we need the complex numbers $\compl$ as well).
For the sake of completeness and readability, 
we have included complete proofs of 
these generalizations,
even in a few cases where 
the arguments are essentially the same.
Thus most of the present paper is logically independent 
of our earlier paper.

We use Theorem~\ref{ithm:gen_syms_char} to answer a 
question of Babai~\cite{Babai77}.
Babai classified finite groups 
which are isomorphic to the euclidean symmetry group  
of a vertex-transitive polytope, i.~e.,
an orbit polytope of a finite orthogonal group.
Moreover, Babai asked which finite groups are isomorphic
to the affine symmetry group of an orbit polytope.
The answer is as follows:

\begin{introthm}\label{ithm:class_aff}
A finite group $G$ is \emph{not} isomorphic to the 
affine symmetry group of an orbit polytope, 
if and only if one of the following holds:
\begin{enumthm}
\item $G$ is abelian of exponent greater than $2$.
\item $G$ is generalized dicyclic.
\item $G$ is elementary abelian of order $4, 8$ or $16$.
\end{enumthm}
In any other case, $G$ is isomorphic to the affine symmetry group
of some orbit polytope.
\end{introthm}

We will recall Babai's classification below 
as Theorem~\ref{thm:Babai77}, as well as the 
definition of generalized dicyclic.
It follows from Theorem~\ref{ithm:class_aff}
and Babai's classification that the only finite groups which
are isomorphic to the euclidean symmetry group
of a vertex-transitive polytope, 
but not to the affine symmetry group of a 
vertex-transitive polytope, 
are the elementary abelian groups of orders $4$, $8$ and $16$.

We will also classify finite groups which are not isomorphic
to the affine symmetry group of a vertex-transitive lattice
polytope,
i.~e. a~polytope with vertices with integer coordinates.

Let us mention that it is probably a folklore result that every
finite group is isomorphic to the affine (or euclidean) symmetry
group of some polytope.
Specifically, a short argument of Isaacs~\cite{Isaacs77} 
can be modified to show
that every finite group is the symmetry group of a polytope
with at most two orbits on the vertices.
More recently, Schulte and Williams showed that every finite group
can be realized as the \emph{combinatorial} symmetry group
of some polytope~\cite{SchulteWilliams15}. 
(A simpler proof has been given by Doignon~\cite{Doignon16pre}.)

Babai's classification is related to the GRR-problem:
a finite group $G$ is said to have 
a \emph{GRR} (\emph{graphical regular representation}),
when there is a graph with vertex set $G$ such that
$G$ is the full automorphism group of this graph
(acting regularly on itself).
The finite groups not admitting a GRR 
have been classified~\cite{Godsil81}.
It follows from this classification and our Theorem~\ref{ithm:class_aff}
that every finite group admitting a GRR is also isomorphic 
to the affine symmetry group of an orbit polytope.
We are not aware of any direct proof of this fact.
(Babai showed directly that every group admitting
a GRR is isomorphic to the euclidean symmetry
group of an orbit polytope.)
Let us also mention that there are exactly $10$
finite groups (up to isomorphism)
which are isomorphic to the affine symmetry group of an orbit 
polytope, but admit no GRR.

Our paper is organized as follows:
In Section~\ref{sec:gen_points}, we carefully introduce 
the definitions of generic points and generic symmetries,
and prove basic results, including 
Theorem~\ref{ithm:gen_pts_zariski}.
In Section~\ref{sec:gen_syms}, 
we consider more closely the relation between
the generic symmetry group as a permutation group
on $G$, and the various linear symmetry groups
$\GL(Gv)$.
Section~\ref{sec:leftideals} contains some results
to compute $\Sym(G,V)$ which are valid in arbitrary 
characteristic.
In Section~\ref{sec:characters}, we specialize to fields
of characteristic zero and show how to compute
$\Sym(G,V)$ from the character of $G$ on $V$
(Theorem~\ref{ithm:gen_syms_char}).
Theorem~\ref{ithm:class_aff} is proved in 
Section~\ref{sec:answerbabai}, and 
Section~\ref{sec:class_q} contains the analogous classification
for vertex-transitive lattice polytopes.

\section{Generic Points}
\label{sec:gen_points}

Throughout, $G$ denotes a finite group, 
$\kk $ a field of infinite order,  
and $V$ a left $\kk G$-module
(thus $G$ acts $\kk $-linearly on $V$).
For a subset $S\subseteq V$
which generates $V$ as $\kk $-vector space,
we write (as in the introduction)
\[ \GL(S)
   := 
   \{ A \in \GL(V) 
      \colon
      A(S) = S
   \}
\]
for the set of linear maps of $V$ which permute
$S$.

We are interested in the various symmetry groups
$\GL(Gv)$ of $G$-orbits $Gv$,
where $v\in \Gen(V)$. Recall that $\Gen(V)$ is the set
of $v\in V$ such that 
\[ 
   V = \kk G v := \Big\{ \sum_{g\in G} c_g gv \colon
                    c_g\in \kk   
                 \Big\} \,.
\]
When there is $v\in V$ such that
$V=\kk Gv$, then $V$ is called \emph{cyclic}
(as $\kk G$-module), and $v$ is called a \emph{generator} of $V$.

In order to compare $\GL(Gv)$ and $\GL(Gw)$
for different $v$, $w\in \Gen(V)$,
we introduce the following definition:

\begin{defn}
Let $v \in V$. 
A permutation $\pi \in \Sym(G)$ is called an 
\emph{orbit symmetry with respect to $v$}, 
if there is a $\kk $-linear map 
from $V$ to $V$
which maps $gv$ to $\pi(g) v$ for all $g \in G$. 
We write $\Sym(G,v)$ for the set of all orbit symmetries of $v$:
\[ \Sym(G,v)
   := \{\pi \in \Sym(G) 
        \colon 
        \exists A \in \GL(V) \colon 
        \forall g\in G \colon  
          Agv = \pi(g)v
      \}\,.
\]
For $\pi \in \Sym(G,v)$, we write
$D_v(\pi)$ for the unique $\kk $-linear map
$\kk Gv \to \kk Gv$ such that
$D_v(\pi)gv = \pi(g)v$ for all $g\in G$.
Then $D_v(\pi)$ is the restriction of $A$ to $\kk Gv$.
\end{defn}

Clearly, the condition $A gv = \pi(g)v$ for all $g\in G$
shows that 
$A(\kk Gv) = \kk Gv$,
and uniquely determines the restriction $D_v(\pi)$ of $A$
to $\kk Gv$.
Conversely, when there is a linear map
$D_v(\pi)\colon \kk Gv \to \kk Gv$
with $D_v(\pi) gv=\pi(g)v$ for all $g\in G$,
then we can extend $D_v(\pi)$ (non-uniquely) to a linear map 
$A\colon V\to V$.
When computing $\Sym(G,v)$, it is thus no loss of generality
to assume $V=\kk Gv$.

For later reference, we record the following easy observation:

\begin{lemma} \label{lm:gen_rep}
$\Sym(G,v)$ is a subgroup of\/ $\Sym(G)$, 
and the map 
\[D_v \colon \Sym(G,v) \to \GL(\kk Gv)
\] 
is a group homomorphism,
and thus a representation of\/ $\Sym(G,v)$.
The image is $D_v(\Sym(G,v)) = \GL(Gv)$.
\end{lemma}
\begin{proof}
    For $\pi$, $\sigma\in \Sym(G,v)$ we have
    \[ D_v(\pi)D_v(\sigma)gv
         = D_v(\pi)\sigma(g)v
         = \pi(\sigma(g))v
         = D_v(\pi\sigma) gv \,.
    \]
    That $D_v(\Sym(G,v)) = \GL(Gv)$ 
    follows directly from the definitions.
\end{proof}

\begin{lemma}\label{lm:rep_kernel}
  Let $H=G_v=\{g\in G \mid gv=v\}$  be the stabilizer of $v$ in $G$.
  Then 
  \begin{align*}
    \Ker D_v &= \{\pi \in \Sym(G) 
                  \colon
                  \pi(gH) = gH \text{ for all cosets } gH
                \} 
          \\ &\iso \Sym(H)^{\card{G:H}} \,.    
  \end{align*}
\end{lemma}
\begin{proof}
  This follows immediately from the definitions.
\end{proof}

It is not difficult to show that 
$\Sym(G,v)$ is in fact isomorphic to a
\emph{wreath product} of $\GL(Gv)$ 
with $\Sym(H)$.
In particular, $\Sym(G,v)$ contains a subgroup
isomorphic to $\Sym(H)^{\card{G:H}}$,
containing ``irrelevant'' permutations.
In view of this, the reader may wonder why we do not simply
consider $\GL(Gv)$ instead of $\Sym(G,v)$.
One reason 
is to make the next definition work:

\begin{defn}
Let $V$ be a cyclic $\kk G$-module, where $\kk $ is an infinite field.
A permutation $\pi \in \Sym(G)$ is called a 
\emph{generic symmetry with respect to}~$V$, 
if it is an orbit symmetry for any generator of~$V$. 
We set
\[\Sym(G,V) := \bigcap_{v\in \Gen(V)} \Sym(G,v)\,,
\] 
the group of all generic symmetries for $V$. 
\end{defn}

Recall that $G$ acts on itself by left multiplication
(the \emph{left regular action}).
For any $h\in G$, let
$\lambda_h\in \Sym(G)$ be the permutation induced by left multiplication
with $h$, 
so $\lambda_h(g)=hg$ for all $g\in G$.
As $V$ is a $\kk G$-module, $G$ acts linearly on $V$, 
say by the representation
$D\colon G\to \GL(V)$.
Since $D(h)(gv) = hgv = \lambda_h(g)v$,
we see that $\lambda_h \in \Sym(G,v)$ for any $v\in V$,
and that $D_v(\lambda_h) = D(h)$.
In particular, $\Sym(G,v)$ and $\Sym(G,V)$
always contain the regular
subgroup $\lambda(G)\iso G$.
This motivates the next definition:

\begin{defn}
The group $\Sym(G,V)$ is called the 
\emph{generic closure} of $G$ with respect to $V$. 
We say that $G$ is \emph{generically closed}
with respect to~$V$ if  $\lambda(G) = \Sym(G,V)$,
where $\lambda\colon G\to \Sym(G)$ is the 
left regular action as above. 
\end{defn}

While these definitions make sense for arbitrary fields,
we will explain below why these
would not be the right definitions for finite fields $\kk $.
We will also indicate how to modify the definitions and 
the results of this section in the case of finite fields. 
However, in this paper, we are mainly interested in fields
of characteristic zero.
For the results in this section, it is enough to assume that
$\kk $ is infinite.

Let us emphasize that
we do \emph{not} assume that $G$ acts faithfully on $V$, that is, 
\[ \Ker(V) := \{g\in G \colon gv = v \text{ for all } v\in V\}
\]
can be non-trivial.
This means that $\Sym(G,V)$ contains by definition
all permutations of $G$ 
which map every left coset of $\Ker(V)$ onto itself.
Of course, when $\pi \in \Sym(G)$ is a permutation
that maps every left coset of $\Ker(V)$ onto itself,
then we have $D_v(\pi) = \id_V$ for every generator $v$, 
and thus the set of these permutations is, in some sense, irrelevant.
This could be avoided by replacing $G$ by the factor group
$G/\Ker(V)$.
It turns out to be more convenient not to do this,
for example in the following situation:

\begin{lemma}\label{lm:gen_sym_sum}
Suppose that $V = V_1 \oplus \dotsb \oplus V_n$
is a direct sum of\/ $\kk G$-modules,
where $V$ is cyclic (that is, $\Gen(V)\neq \emptyset$). 
Then
\[ \bigcap_{i=1}^n \Sym(G,V_i)
   \subseteq \Sym(G,V)\,.
\]
\end{lemma}

Notice that it is perfectly possible that $\Ker(V)=1$,
while $\Ker(V_i)\neq 1$ for some $V_i$.

\begin{proof}[Proof of Lemma~\ref{lm:gen_sym_sum}]
  Let $v\in \Gen(V)$ and write
  $v=v_1 + \dotsb + v_n$ with $v_i\in V_i$.
  Then $v_i \in \Gen(V_i)$.
  Suppose that $\pi\in \Sym(G,V_i)$ for all $i$.
  Thus there is $A_i\in \GL(V_i)$ such that
  $\pi(g)v_i = A_i gv_i$ for all $g\in G$.
  Then for $A = A_1 \oplus \dotsb \oplus A_n\colon V\to V$,
  we have $\pi(g)v = Agv$ for all $g\in G$,
  and thus $\pi \in \Sym(G,V)$ as claimed.
\end{proof}

We will show later that
when $\kk $ has characteristic zero, 
then there is a certain decomposition
such that equality holds in Lemma~\ref{lm:gen_sym_sum}.
In general, the containment is of course strict.

\begin{rmk} \label{rmk:triv_sym}
When $1< \Ker(V) < G$ or $\card{\Ker(V)}\geq 3$, 
then there is a permutation $\pi\neq \id_G$ of $G$
which maps every coset of $\Ker(V)$ onto itself, 
and such that $\pi(1_G) = 1_G$.
Then $\pi \in \Sym(G,V)$, but $\pi$ is not of the form 
$\pi = \lambda_g$ for any $g \in G$. 
Hence, when $G$ is generically closed with respect to~$V$, 
then $G$ must act faithfully on $V$
(except in the trivial case $G=C_2$).
\end{rmk}

\begin{defn}\label{df:generic}
A point $v \in V$ is called a \emph{generic point} of $V$,
when $v$ is a generator of $V$,
when $G_v = \Ker(V)$ (the stabilizer has minimal possible size),
and $\Sym(G,v) = \Sym(G,V)$ 
(the symmetry group of the orbit $Gv$ has minimal possible
size).
\end{defn}

One can show that $G_v = \Ker(V)$ follows from the other conditions.
As we will see below,
the above definition is consistent with 
the definition of generic points
from our previous paper~\cite[Definition~4.4]{FrieseLadisch16},
given in the special situation $G\leq \GL(d, \reals)$
and $V= \reals^d$.

The term “generic point” is justified by the result that
“almost all” points of $V$ satisfy this property
(see Theorem~\ref{thm:existence_generic} below,
which contains Theorem~\ref{ithm:gen_pts_zariski}
from the introduction).

\begin{lemma}\label{lm:iso_inv}
   Let $\phi\colon V \to W$ be an isomorphism of\/
   $\kk G$-modules
   (i.~e., $\phi$ is $\kk $-linear, bijective,
    and $\phi(gv)=g\phi(v)$
   for $g\in G$, $v\in V$). Then:
   \begin{enumthm}
   \item $\Sym(G,v) = \Sym(G,\phi(v) )$ 
         for any $v\in V$.
   \item $D_{\phi(v)}(\pi) = \phi\circ D_v(\pi) \circ \phi^{-1}$
         for $v\in V$ and $\pi \in \Sym(G,v)$.
   \item $v\in V$ is generic if and only if $\phi(v)$ is generic.
   \item $\Sym(G,V) = \Sym(G,W)$.
   \end{enumthm}
\end{lemma}
\begin{proof}
   Easy verifications.
\end{proof}

In view of this lemma, it is no loss of generality
to assume that $V = \kk^d$ as $\kk $-space, 
so that the action of $G$ on $V$ is described by a
matrix representation $D \colon G \to \GL(d,\kk )$. 
We do this in the rest of this section.
In particular, this enables us to evaluate polynomials
in $d$ indeterminates at elements $v\in V$
in the usual, elementary way.
We view $V= \kk^d$ as equipped with the 
Zariski topology, that is, the closed subsets of $V$
are by definition the zero sets of arbitrary families of
polynomials in $\kk [X_1, \dotsc, X_d]$.
(We mention in passing that any finite dimensional 
$\kk $-space can be equipped 
with the Zariski topology by choosing a basis,
and that the resulting topology does not depend on the choice 
of basis.)

Since $\kk $ is infinite by assumption, 
$V=\kk^d$ is irreducible as a topological space,
i.~e. the intersection of any two non-empty open subsets is non-empty
as well.
Equivalently, any non-empty open subset is dense in $V$. 
We are now going to show that the set of generic points of $V$
is non-empty and open,
when $V=\kk^d$ is a cyclic $\kk G$-module and $\kk $ is infinite.
In our previous paper, we proved this in the case 
$\kk = \reals$~\cite[Corollary~4.5]{FrieseLadisch16}.

\begin{lemma} \label{lm:gen_open}
$\Gen(V)$ is open in $V$.
\end{lemma}
\begin{proof}
Let $X = (X_1, \dots, X_d)^t$ be a vector of indeterminates.
A vector $v\in \kk^d$ is a generator, if and only if 
$\kk^d$ is the $\kk $-linear span of the vectors
$gv$, where $g$ runs through $G$.
This is the case if and only if $Gv$ contains a basis
$\{g_1v, \dotsc, g_dv \}$ (say) of $\kk^d$.
This means that $f(v)\neq 0$, where
$f(X)= \det(g_1X, \dotsc, g_d X)$.
Thus
$\Gen(V)$ is the union of the non-vanishing sets
$O_f:= \{w\in V \colon f(w)\neq 0\}$ of~$f$,
where $f$ runs through the 
$d\times d$ sub-determinants of the matrix with columns
$gX$, $g\in G$.
\end{proof}

\begin{lemma} \label{lm:triv_stab_open}
The set of points $v \in V$ with $G_v > \Ker(V)$ is a finite union
of proper subspaces of $V$. In particular, it is a closed proper
subset of $V$.
\end{lemma}
\begin{proof}
For $g \in G \setminus \Ker(V)$, the fixed space
$\operatorname{Fix}(g) = \{ v \in V : gv = v \}$ is a proper
subspace of $V$. The set of points $v$ with $G_v > \Ker(V)$ is
given by $\bigcup_{g \in G \setminus \Ker(V)}
\operatorname{Fix}(g)$.
\end{proof}

\begin{lemma} \label{lm:sym_points_closed}
Let $\pi \in \Sym(G)$. 
Then the set of all points $v \in \Gen(V)$ with 
$\pi \in \Sym(G,v)$ 
is relatively closed in $\Gen(V)$.
\end{lemma}
\begin{proof}
Define
\[ Q(\pi) :=
   \{ v\in \Gen(V) 
      \colon
      \pi \in \Sym(G,v)
   \} \,.
\]
Let $g_1$, $\dotsc$, $g_d\in G$ be elements such that
$\{g_1v, \dotsc, g_dv\}$ is a basis of $V$ for some $v\in V$.
Equivalently, 
$f(X) := \det(g_1 X, \dotsc, g_d X)$,
where $X = (X_1, \dots, X_d)^t \in \kk (X)^d$,
is not the zero polynomial.
Since $\Gen(V)$ is the union of the non-vanishing sets $O_f$,
where $f$ runs through the polynomials constructed in the above way,
it suffices to show that $Q(\pi) \cap O_f$ is closed in $O_f$.

Since $f(X)\neq 0$, the matrix
\[  A := A(X) 
      := (\pi(g_1) X, \dotsc,\pi(g_d) X) \cdot 
         (g_1 X, \dotsc, g_d X)^{-1}
\] 
is defined and has entries in the function field $\kk (X)$.
Moreover, for $v\in O_f$, we can evaluate $A(X)$ at $v$,
and $A(v)$ is the unique matrix mapping the basis vectors
$g_iv$ to $\pi(g_i)v$.
It follows that for $v\in O_f$, we have
$\pi \in \Sym(G,v)$ if and only if
$A(v)gv = \pi(g)v$ for \emph{all} $g\in G$.
Therefore, 
$Q(\pi)\cap O_f$
is exactly the common vanishing set in $O_f$ 
of all the entries of all the vectors
$A(X) g X - \pi(g) X$ ($g \in G$).
Since the entries of 
$f(X)\big( A(X) g X - \pi(g) X \big)$ are polynomials,
this finishes the proof.
\end{proof}

\begin{rmk}\label{rmk:gen_mat_eval}
  In the proof, we defined a $d\times d$-matrix $A(X)$ 
  with entries in $\kk (X)$,
  depending on the choice of 
  $d$ elements $g_1$, $\dotsc$, $g_d\in G$ such that
  \[ f(X) = \det( g_1 X, \dotsc, g_dX)\neq 0 \,.
  \]
  This matrix has the following property:
  For any $v\in Q(\pi)\cap O_f$, we have
  $A(v) = D_v(\pi)$.
  In particular, $A(v)$ is invertible, and so
  when $Q(\pi)\cap O_f \neq \emptyset$, 
  then $A(X)$ must be invertible.
\end{rmk}

\begin{rmk}
  The set of $v\in V$ such that $\pi \in \Sym(G,v)$
  is in general not a closed subset of $V$ itself.
  For an example, let 
  $G = D_4$ be the dihedral group of order $8$,
  and represent $G$ as the subgroup of 
  $\GL(2,\reals)$ preserving a (fixed) square 
  which is centered at the origin.
  Let $V$ be the space of $2\times 2$-matrices over $\reals$,
  on which $G$ acts by left multiplication.
  Let $\pi\in \Sym(G)$ be the permutation sending each element
  to its inverse.
  Then it is not difficult to verify that
  $\{v\in V \colon \pi \in \Sym(G,v)\}
   = \Gen(V) \cup \{0\}$.
  (This is also a consequence of a general result
   about representation polytopes~\cite[Theorem~8.6]{FrieseLadisch16}.)
\end{rmk}

\begin{thm} \label{thm:existence_generic}
Let $V$ be a cyclic $\kk G v$-module,
where $\kk $ is infinite.
The set of generic points is non-empty and open 
(in the Zariski topology).
In particular, there are generic points.
\end{thm}
\begin{proof}
This is immediate from
Lemmas~\ref{lm:gen_open}, \ref{lm:triv_stab_open}
and~\ref{lm:sym_points_closed}.
\end{proof}

Of course, the last result contains Theorem~\ref{ithm:gen_pts_zariski}
form the introduction.

As before, let $G$ act linearly on $\kk^d$, 
by some matrix representation $D\colon G\to \GL(d,\kk )$.
Then the same representation makes $\EE^d$ into an
$\EE G$-module, for every field extension
$\EE$ of $\kk $.
We apply this to the function field $\EE = \kk (X)$,
where $X = (X_1, \dots, X_d)^t$ is a vector of indeterminates.
Not surprisingly, $X$ is a generic vector in $\kk (X)^d$:

\begin{lemma} \label{lm:gen_closure_char}
Let $G$ act on $V = \kk^d$ and $\kk (X)^d$ 
by a representation
$D\colon G\to \GL(d,\kk )$, 
where $X = (X_1, \dots, X_d)^t$ is a vector of indeterminates. 
Then\/
\[ \Sym(G,V) = \Sym(G,X) \,.
\]
\end{lemma}

This lemma also shows that Definition~\ref{df:generic}
is equivalent to the definition of generic points
from our previous paper~\cite[Definition~4.4]{FrieseLadisch16}.

\begin{proof}[Proof of Lemma~\ref{lm:gen_closure_char}]
We first show 
$\Sym(G,V) \subseteq \Sym(G,X)$.
Let $\pi \in \Sym(G,V)$, 
and let $f(X)= \det(g_1X, \dotsc, g_d X)$ and $A(X)$ be as in
Remark~\ref{rmk:gen_mat_eval}. 
It follows from this remark that
$A(v) = D_v(\pi)$ for all $v \in O_f$,
so $A(v) gv = \pi(g)v$ for all $v\in O_f$.
Since $O_f$ is non-empty and open, 
we have $A(X)  gX = \pi(g)  X$ for all $g \in G$, 
which means $\pi \in \Sym(G,X)$.

Conversely, suppose $\pi \in \Sym(G,X)$ is realized by a matrix 
$A(X) := D_X(\pi) \in \GL(d,\kk (X))$. 
The set $O$ of generators on which $A(X)$ can be evaluated,
is non-empty and open
($O$ is the non-vanishing set of a common denominator
of all the entries of $A(X)$).
Moreover, $\pi \in \Sym(G,v)$ for all $v \in O$. 
By Lemma~\ref{lm:sym_points_closed}, 
$\pi$ has to be an orbit symmetry for all elements 
in the closure of $O$ in $\Gen(V)$. 
Since $\Gen(V)$ is irreducible (as topological space), 
this closure is equal to $\Gen(V)$, which shows 
$\pi \in \Sym(G,V)$.
\end{proof}

The following proposition will be an important tool in 
Section~\ref{sec:characters}. 
It means that the generic symmetry group of a 
$\kk G$-module does not change if we extend the field. 
In particular, we are always allowed to assume that $\kk $ is
algebraically closed.

\begin{prp}\label{prp:scalar_ext}
Let\/ $\EE$ be an extension field of\/ $\kk $, 
and let $V $ be a cyclic $\kk G$-module.
Then $\Sym(G,V) = \Sym(G,V \otimes_{\kk } \EE)$.
\end{prp}
\begin{proof}
Without loss of generality, we can assume that
$V=\kk^d$ and $V\otimes_{\kk }\EE = \EE^d$.
Let $X=(X_1,\dotsc, X_d)^t$ be a vector of indeterminates
over $\EE$.
By Lemma~\ref{lm:gen_closure_char} applied first to $\kk^d$,
then to $\EE^d$, we have,
$\Sym(G,V) = \Sym(G,X)= \Sym(G, V\otimes_{\kk } \EE)$.
\end{proof}

\begin{rmk}
  In view of the above results, it seems natural to define
  generic symmetries and generic points for possibly finite
  fields as follows:
  A point $v\in V$ is generic, when
  $v\in \Gen(V)$ and $G_v = \Ker(V)$, and
  when $\Sym(G,v)= \Sym(G,X)$,
  where $X$ is a vector of indeterminates.  
  The generic symmetry group can be defined 
  as the group $\Sym(G,X)$
  or equivalently as the group
  $\Sym(G, V\otimes_{\kk } \EE)$,
  where $\EE$ is ``sufficiently large''
  (e.~g., infinite).
  When $\kk $ is finite, then $V$ may not contain
  generic points in this sense, 
  but $V\otimes_{\kk } \overline{\kk }$ does,
  where $\overline{\kk }$ is the algebraic closure of $\kk $.  
\end{rmk}

\section{The generic symmetry group}
\label{sec:gen_syms}
Recall the notations
\[\GL(Gv) 
 := D_v(\Sym(G,v)) \subseteq \GL(d,\kk )
\]
and 
\[\GL(GX) 
          := D_X(\Sym(G,X)) 
          \subseteq \GL(d,\kk (X))\,.
\]
\begin{lemma}\label{lm:eval_gen} 
For $\pi\in \Sym(G,V) = \Sym(G,X)$,
the matrix $A(X):= D_X(\pi) \in \GL(d,\kk (X))$
can in fact be evaluated for all $v\in \Gen(V)$,
and evaluates to $A(v) = D_v(\pi)$.
Thus we have a commutative diagram:
\[ \begin{tikzcd}
     \Sym(G,V) \rar{D_X} 
               \drar[swap]{D_v}
      & \GL(GX) \dar{\operatorname{eval}_v} 
      \\
      & \GL(Gv)
   \end{tikzcd} 
\]
\end{lemma}
\begin{proof}
  Let $v$ be a generating point, and let
  $g_1$, $\dotsc$, $g_d\in G$ be elements such that
  $\{g_1 v,\dotsc, g_d v\}$ is a basis of $\kk^d$.
  Then $\{g_1 X, \dotsc, g_d X\}$ is a basis of
  $\kk (X)^d$,
  and we must have
  \[A(X) = (\pi(g_1)X, \dotsc, \pi(g_d)X) \cdot 
          (g_1 X, \dotsc, g_d X)^{-1}\,.
  \]
  As $f(v)\neq 0$, where
  $f(X)= \det(g_1 X, \dotsc, g_d X)$, 
  it follows that $A(X)$ can be evaluated at $v$.
  Also, $A(v)=D_v(\pi)$ is clear then.
  Since $v \in \Gen(V)$ was arbitrary, the claim follows.
\end{proof}  

It is clear that the map $\GL(GX)\to \GL(Gv)$
is an isomorphism when $v$ is generic.
Somewhat more is true.

\begin{lemma}\label{lm:inj_rep_gen_sym}
  Let $v\in \Gen(V)$ be such that the characteristic
  of\/ $\kk $ does not divide the order of the stabilizer
  $H=G_v$ of $v$ in $G$.
  Then evaluation at $v$ yields an injective map
  $\GL(GX) \to \GL(Gv)$. 
\end{lemma}
\begin{proof}
  Suppose that $A(X)\in \GL(GX)$ evaluates to the identity.
  Thus $A(v) gv = gv$ for all $g\in G$.
  This means that $A(X)$ maps the set 
  $gHX$ onto itself.
  Define 
  \[ s_g(X) := \frac{1}{\card{H}} 
               \sum_{h\in H}  ghX \in \kk [X]^d\,.
  \]
  (Here we need that $\card{H}$ is invertible
   as an element of $\kk $.)
  Then $A(X)s_g(X) = s_g(X)$ and
  $s_g(v) = gv$.
  As $V= \kk^d$ is the $\kk $-linear span of the elements
  $s_g(v)=gv$ ($g\in G$),
  it follows that
  $\kk (X)^d$ is the $\kk (X)$-linear span of the elements
  $s_g(X)$ ($g\in G$).
  Since $A(X)s_g(X) = s_g(X)$ for all $g$,
  it follows that $A(X) = I$ as claimed.
\end{proof}

Let $\widehat{G}=\GL(Gv)$,
where $Gv$ spans $V$.
Then we can view $V$ as a $\kk \widehat{G}$-module,
and we can speak of generic points for $\widehat{G}$.
In our previous paper, we showed that
when $\kk =\reals$ and $w$ is generic for $\widehat{G}$,
then 
$\GL(\widehat{G}w) = \widehat{G}$
\cite[Corollary~5.4]{FrieseLadisch16}.
In particular, we can not get an infinitely increasing
chain of generic symmetry groups.
This can be generalized as follows:
\begin{co}\label{co:closure}
  Let $\widehat{G}=\GL(Gv)$, 
  where $v\in \Gen(V)$ (with respect to the action of $G$),
  and let $w\in V$ be generic for $\widehat{G}$.
  If the characteristic of\/ $\kk $ does not divide
  $\card{\widehat{G}_v}$, then
  $\widehat{G} = \GL(\widehat{G}w)$.
\end{co}
\begin{proof}
  By Lemma~\ref{lm:eval_gen} applied to $\widehat{G}$ and $w$,
  it follows that 
  $\GL(\widehat{G}X) \iso \GL(\widehat{G}w)$ 
  by evaluation at $w$.
  (Since $w$ is generic for $\widehat{G}$,
  we have that 
  $\Sym(\widehat{G},V) = \Sym(\widehat{G}, w)$.)
  By Lemma~\ref{lm:inj_rep_gen_sym} applied to
  $\widehat{G}$ and $v$, it follows that
  $\GL(\widehat{G}X)$ maps injectively into
  $\GL(\widehat{G}v)$.
  But by definition of $\widehat{G}$,
  we have $\widehat{G}v = Gv$ 
  and $\GL(Gv) = \widehat{G}$.
  Thus $\GL(\widehat{G}w)\iso \GL(\widehat{G}X)$ is isomorphic to
  a subgroup of $\widehat{G}$.
  On the other hand, $\widehat{G}\leq \GL(\widehat{G}w)$.
  The result follows.
\end{proof}

We now digress to give an example
which shows that 
the conclusions of
Lemma~\ref{lm:inj_rep_gen_sym}
and Corollary~\ref{co:closure}
may fail to hold if the characteristic of $\kk $
divides the order of the stabilizer.

\begin{expl}\label{expl:char2_incr_gen}
  Let $\kk $ be a field of characteristic~$2$.
  Let $U \leq \kk^2$ be a finite additive subgroup
  such that 
  $(u,v)\in U$ implies
  $(0,u)\in U$. 
  This condition ensures that
  \[
    G := 
      \left\{ 
        \begin{pmatrix}
           1 & u & v \\
             & 1 & c \\
             &   & 1
        \end{pmatrix}
        \colon
        (u,v) \in U,
        c\in \GF{2}
      \right\}
  \]
  is a finite subgroup of $\GL(3,\kk )$.
  Moreover, we assume that
  \[ \{\lambda \in \kk 
       \colon
       \lambda U \subseteq U
     \}
     = \{0,1\} = \GF{2}
  \]
  and that
    $(\GF{2})^2 \subseteq U$.
  (For example, we can choose $U=(\GF{2})^2$.)
  A vector $(x,y,z)^t$ is a generator of 
  $\kk^3$ if and only if $z\neq 0$.
  
  Let $W = (X,Y,Z)^t \in \kk (X,Y,Z)^3$.
  It is easy to check that each element of
  \[ H = 
     \left\{
       \begin{pmatrix}
         1 & (uY +vZ)/Z & (uY+vZ)Y/Z^2 \\
           & 1     & 0 \\
           &       & 1
       \end{pmatrix}
       \colon (u,v)\in U
     \right\}
  \]
  maps the orbit $GW$ onto itself, and fixes $W$.
  For example, for $(u,v) = (1,1)$, we get the matrix
    \[ A(X,Y,Z) = 
       \begin{pmatrix}
         1 & Y/Z + 1 & (Y/Z +1)Y/Z \\
           & 1       & 0 \\
           &         & 1
       \end{pmatrix} \in \GL(GW)\,.
    \]
    On the other hand, we have $A(1,1,1) = I$, 
    and so evaluation is not injective in this case.
   Lemma~\ref{lm:inj_rep_gen_sym} does not apply here
   since $2$ (the characteristic of $\kk $)
   divides the order of the stabilizer of $(1,1,1)^t$ 
   in $G$.
  
  It is somewhat tedious, but elementary, to compute 
  that $H$ is in fact exactly the set of matrices that fix 
  the generic vector $W$, and map
  its orbit $GW$ onto itself.
  (Here we need that $\lambda U \subseteq U$ implies
    $\lambda \in \{0,1 \}$.)  
  Since $G$ acts regularly on $GW$, it follows
  that $\GL(GW) = HG > G$.
  
  Now suppose $w = (x,y,z)^t \in \kk^3$ is a generic vector.
  (Recall that generic vectors exist when $\kk $ is large enough,
   which we simply assume now.)
  It follows that $\widehat{G}:= \GL(Gw)$
  has also the form
  \[
    \widehat{G}
    = 
       \left\{ 
         \begin{pmatrix}
            1 & u & v \\
              & 1 & c \\
              &   & 1
         \end{pmatrix}
         \colon
         (u,v) \in \widehat{U},
         c\in \GF{2}
      \right\}\,,
  \]
  with a finite subgroup $\widehat{U}\leq \kk^2$
  such that $U < \widehat{U}$.
  If $\widehat{U}$ also fulfills the assumption
  that $\lambda \widehat{U}\subseteq \widehat{U}$ 
  implies $\lambda \in \GF{2}$,
  then we can continue as before.
  For example, when $\kk =\GF{2}(t)$
  (the function field in one variable),
  this will be true automatically
  (as every $\lambda \in \GF{2}(t) \setminus \GF{2}$
   has infinite order, but
   $\widehat{U}$ is finite).
  Thus we can start with $U=(\GF{2})^2$,
  and we get an infinitely increasing chain of 
  generic symmetry groups.
\end{expl}

By Lemma~\ref{lm:gen_rep}, 
any generating point 
$v\in \Gen(V)$ 
defines a representation 
\[D_v \colon \Sym(G,v) \to \GL(V) \,.
\] 
We now consider the restrictions
to the generic symmetry group,
$\Sym(G,V)$.

\begin{lemma} \label{lm:gen_rep_eq}
  The character of the restriction
  $D_v\colon \Sym(G,V) \to \GL(V)$ is independent
  of $v\in \Gen(V)$.
\end{lemma}
\begin{proof}
Let $\pi \in \Sym(G,V)$ be a generic symmetry,
and let 
$A(X) = D_X(\pi) \in \GL(GX)$ be the matrix realizing 
$\pi$ as an orbit symmetry of 
the vector of indeterminates $X\in \kk (X)^d$. 
By Lemma~\ref{lm:eval_gen}, $A(X)$ evaluates to 
$A(v) = D_v(\pi)$ for any $v\in \Gen(V)$.
Thus the rational function
$ f(X)= \Tr(A(X)) \in \kk (X)$ evaluates to
$f(v) = \Tr(D_v(\pi))$. 
On the other hand,
$A(X)=D_X(\pi)$ has finite order 
and thus $\Tr(A(X))$ is a sum of roots of unity.
Thus $f(X)$ is algebraic over $\kk $. 
Since $\kk (X)/\kk $ is purely transcendental, 
we conclude that 
$f(X)\in \kk $,
which means that 
$f(v)=\Tr(D_v(\pi))$ is independent of $v$.
\end{proof}

As in our earlier paper \cite[Theorem 5.3]{FrieseLadisch16},
it follows that when $\kk $ has characteristic zero,
then the different $D_v$'s are similar
as representations of $\Sym(G,V)$.
This may be wrong in positive characteristic,
as Example~\ref{expl:char2_incr_gen} shows.

\begin{prp} \label{prp:abs_simp_gen_sym}
Let $D \colon G \to \GL(V)$ be an absolutely irreducible 
representation. 
Then 
\[ D_v(\pi) = D(\pi(1))
   \quad \text{for all $v\in \Gen(V)$ and $\pi\in \Sym(G,V)$}\,.
\]
In particular,
$\GL(Gv) = D(G)$ for every generic point $v\in V$.
\end{prp}

This result generalizes another result from our previous
paper~\cite[Theorem 5.5]{FrieseLadisch16} 
to arbitrary (infinite) fields.
In fact, an old paper of Isaacs already contains the conclusion
that $\GL(Gv) = D(G)$ 
for \emph{some} point $v\in V$~\cite{Isaacs77}.
Our formulation here is chosen with a view to later
applications.

\begin{proof}[Proof of Proposition~\ref{prp:abs_simp_gen_sym}]
Let $v,w \in \Gen(V)$ be arbitrary. 
By Lemma~\ref{lm:gen_rep_eq},
the representations 
\[ D_v, D_w \colon \Sym(G,V) \to \GL(V)
\]
have the same character. 
The representations $D_v$ and $D_w$ are absolutely irreducible,
as $D$ is absolutely irreducible.
Since the character determines an irreducible representation 
up to equivalence~\cite[Corollary~9.22]{isaCTdov}, 
$D_v$ and $D_w$ are equivalent.
Thus there is a linear map
$\varphi \in \GL(V)$ with 
$\varphi \circ D_v(\pi) = D_w(\pi) \circ \varphi$ 
for all 
$\pi \in \Sym(G,V)$. 
In particular, 
$\varphi \circ D_v(\lambda_g) = D_w(\lambda_g) \circ \varphi$ 
for all 
$g \in G$, where $\lambda_g$ is the left multiplication by $g$.
Since $D_v(\lambda_g) = D(g) =  D_w(\lambda_g)$, 
this shows that
$\varphi \in \End_{\kk G}(V) = \kk \cdot \mathrm{id}_V$. 
Hence $D_v = D_w$.

It follows that the action of $\Sym(G,V)$ 
on $V$ via $D_v$
is in fact independent of
$v \in \Gen(V)$.
Thus we can pick some $w\in \Gen(V)$ which is 
generic for $\Sym(G,V)$.
Let $\pi\in \Sym(G,V)$ and set $g=\pi(1)$.
Then $D_w(\lambda_g^{-1}\pi)w = w$,
and since $w$ is generic for $\Sym(G,V)$,
we have $\lambda_g^{-1} \pi\in \Ker D_w$.
Hence
$D_v(\pi)=D_w(\pi)=D_w(\lambda_g) = D(g)$, which is the first claim.

In particular, for $v$ generic for $G$, we have
$\GL(Gv)= D_v(\Sym(G,V))= D(G)$.
\end{proof}

\section{Generic symmetries and left ideals}
\label{sec:leftideals}
In the following, we will characterize generic symmetries in terms of
left ideals of the group algebra $\kk G$. 
For a left $\kk G$-module $V$ and $v\in V$,
we set
\[ \Ann(v) := \Ann_{\kk G}(v)
           := \{ a\in \kk G \colon av = 0
              \}\,,
\]
the annihilator of $v$ in $\kk G$.
This is a left ideal of $\kk G$.

Note that $G$ is a basis of $\kk G$, 
and so any permutation $\pi \in \Sym(G)$ uniquely
extends to an automorphism of the $\kk $-vector space $\kk G$, 
which we will also denote by $\pi$.

\begin{lemma} \label{lm:sym_by_ann}
Let $v \in \Gen(V)$ and $\pi \in \Sym(G)$. 
Then $\pi$ is an orbit
symmetry for $v$ if and only if 
$\pi(\Ann(v)) \subseteq \Ann(v)$.
\end{lemma}
\begin{proof}
Let $\kappa_v \colon \kk G\to V$ be the map defined by
$\kappa_v(a) = av$.
This is a homomorphism of left $\kk G$-modules with kernel
$\Ann(v)$.

By definition, $\pi$ is an orbit symmetry for $v$, 
if and only if there is a linear map
$\alpha \colon V \to V$, 
such that $\alpha(gv) = \pi(g)v$ for all $g\in G$.
This means that $\alpha$ makes the following diagram commute:
\[  \begin{tikzcd}
       \kk G \rar{\pi} 
            \dar{\kappa_v} 
         & \kk G \dar{\kappa_v}
         \\
        V \rar[dashed]{\alpha}
         & V
    \end{tikzcd}
\]
Since $\kappa_v$ is surjective (because $V=\kk Gv$),
such an $\alpha$ exists if and only if
$\Ann(v) = \Ker \kappa_v \subseteq \Ker(\kappa_v \circ \pi)$.
The last equality is equivalent to
$\pi(\Ann(v)) \subseteq \Ann(v)$, as $\pi$ is invertible.
\end{proof}

\begin{lemma} \label{lm:gen_sym_by_ann}
Let $\pi \in \Sym(G)$ and $v \in \Gen(V)$, 
and set $L := \Ann(v)$.
Then the following are equivalent:
\begin{enumthm}
\item \label{it:gen_sym} 
      $\pi \in \Sym(G,V)$,
\item \label{it:ann_unit}
      $\pi(Ls) \subseteq Ls$ for all units $s\in (\kk G)^{\times}$,
\item \label{it:ann_iso}
      $\pi(\widetilde{L}) \subseteq \widetilde{L}$ 
      for every left ideal $\widetilde{L}$
      that is isomorphic to $L$ (as left $\kk G$-module).
\end{enumthm} 
\end{lemma}
\begin{proof}
  We begin with ``\ref{it:gen_sym}~$\implies$~\ref{it:ann_iso}''.
  Let $\pi \in \Sym(G,V)$ and 
  assume that $L\iso \widetilde{L}$ as left $\kk G$-modules.
  We claim that also $\kk G/L \iso \kk G/\widetilde{L}$
  (as left $\kk G$-modules).
  This is clear if $\kk G$ is semisimple
  (which is the only case where we will apply this lemma),
  but is also true for 
  Frobenius rings~\cite[Theorem~15.21]{Lam99},
  and $\kk G$ is 
  a Frobenius ring~\cite[Example~3.15E]{Lam99}.
  Thus $V = \kk Gv \iso \kk G/L \iso \kk G/\widetilde{L}$,
  and $\Sym(G,V) = \Sym(G, \kk G/\widetilde{L})$.
  As $\widetilde{L}$ is the annihilator of 
  $1 +\widetilde{L}$ in $\kk G$,
  Lemma~\ref{lm:sym_by_ann} yields that
  $\pi(\widetilde{L})\subseteq \widetilde{L}$.
     
  That \ref{it:ann_iso} implies~\ref{it:ann_unit}
  is clear since $Ls \iso L$.
  
  Now assume \ref{it:ann_unit},
  and let $w\in \Gen(V)$ be another generator.  
  By a theorem of Bass~\cite[20.9]{Lam01},
  it follows that
  $v = s w$ for a \emph{unit} $s\in (\kk G)^{\times}$. 
  Thus $\Ann(w) = \Ann(v)s = Ls$.
  Then Lemma~\ref{lm:sym_by_ann} yields that $\pi \in \Sym(G,w)$,
  and thus $\pi\in \Sym(G,V)$.
\end{proof}

Let us mention in passing
that in a Frobenius ring, every left
ideal isomorphic to $L$ is of the form $Ls$
with some unit $s$ \cite[Proposition~15.20]{Lam99}.

\begin{co}
  Suppose that $\Ann(v)$ is a (twosided) ideal of\/ $\kk G$,
  where $v\in \Gen(V)$.
  Then $\Sym(G,w) = \Sym(G,v)$ for all $w\in \Gen(V)$,
  and in fact all $w\in \Gen(V)$ are generic.
\end{co}
\begin{proof}
  The first assertion is immediate from 
  Lemma~\ref{lm:gen_sym_by_ann} and Lemma~\ref{lm:sym_by_ann}.
  The stabilizer in $G$ of a point $v$ 
  is the set of $g\in G$ such that $g-1 \in \Ann(v)$,
  and $\Ann(w) = \Ann(v)s=\Ann(v)$ for all $w\in \Gen(V)$.
  Thus all $w\in \Gen(V)$ are generic.
\end{proof}

Although very simple, 
Lemma~\ref{lm:gen_sym_by_ann} has quite remarkable consequences.
For example, when $\pi$ is a generic symmetry for 
the cyclic modules $\kk G/ L_1$ and $\kk G/L_2$,
where $L_1$ and $L_2$ are left ideals,
then it is immediate from the characterization
in Lemma~\ref{lm:gen_sym_by_ann}
that $\pi$ is also generic for the modules
$\kk G/(L_1\cap L_2)$ and $\kk G/(L_1 + L_2)$.

Also, when $\pi $ is generic for $\kk G/L$, and $I$
is any left ideal which we get by 
repeatedly taking intersections
and sums of left ideals isomorphic to $L$,
then $\pi$ is generic for $\kk G /I$.
For example, we can take for $I$ the sum of all left 
ideals isomorphic to $L$.
This will be used below in the case where
$\kk $ has characteristic zero.

\section{Character Criteria}
\label{sec:characters}

In this section, we work over the field $\compl$ of complex numbers.
The aim of this section is to describe 
the generic symmetries of some (cyclic) $\compl G$-module
$V$ in terms of its character, $\chi$,
and in particular, 
its decomposition into irreducible characters.

We emphasize that instead of $\compl$, any field
$\kk $ of characteristic zero would do, for the following reasons:
Suppose that $V$ is a cyclic $\kk G$-module, where
$\kk $ has characteristic zero.
By Proposition~\ref{prp:scalar_ext}, we have
$\Sym(G,V) = \Sym(G, V \otimes_{\kk } \overline{\kk })$,
where $\overline{\kk }$ is 
the algebraic closure of $\kk $.
But over an algebraically closed field, any representation
is similar to a representation with entries
in $\overline{\QQ}$, the algebraic closure of 
the rational numbers $\QQ$
(which embeds into $\overline{\kk }$).
This means that there is a module $V_0$ over
$\overline{\QQ} G$ such that
$V\otimes_{\kk } \overline{\kk }
 \iso V_0 \otimes_{ \overline{\QQ} } \overline{\kk }$.
By Lemma~\ref{lm:iso_inv} and Proposition~\ref{prp:scalar_ext},
we have 
\[ \Sym(G,V) 
     = \Sym( G, V \otimes_{\kk } \overline{\kk } )
     = \Sym( G, V_0 \otimes_{\overline{\QQ}} \overline{\kk } )
     = \Sym( G, V_0 )\,.
\]
Thus we can assume without loss of generality that
$\kk = \overline{{\QQ}}$ or, 
as is more conventional, that $\kk =\compl$.

Since we are in characteristic zero,
any $\compl G$-module $V$ is determined up to
isomorphism by its character 
$\chi\colon G\to \compl$,
$\chi(g) = \Tr_V(g)$.
This suggests the first part of the following definition:

\begin{defn}
Let $\chi$ be a character
which is afforded by the cyclic $\compl G$-module~$V$. 
Then we set $\Sym(G,\chi) := \Sym(G,V)$.
The character of the representation 
\[ D_v \colon \Sym(G,\chi) \to \GL(V) 
\]
from Lemma~\ref{lm:gen_rep},
where $v$ is any generator of $V$, is denoted by $\widehat{\chi}$.
\end{defn}
Note that $\widehat{\chi}$ does not depend
on the choice of $v\in V$, by Lemma~\ref{lm:gen_rep_eq}.
By Lemma~\ref{lm:iso_inv},
$\Sym(G,\chi)$ and $\widehat{\chi}$ do not depend on
the choice of the module $V$ itself.
More generally, if $\chi$ is afforded by some module 
$\widetilde{V}$ over $\kk G$ for some other field $\kk $,
then $\Sym(G,\chi)$ and $\widehat{\chi}$ can also 
be defined with respect to $\widetilde{V}$,
by the remarks above.

Also note that
via the left regular action
$\lambda\colon G \to \Sym(G,V)$, we can view 
$\widehat{\chi}$ as an extension of $\chi$.

We call $\Sym(G,\chi)$ the generic symmetry group of $\chi$.
Likewise, we say that $\pi \in \Sym(G)$ is a generic symmetry 
for $\chi$ or that $G$ is
generically closed with respect to $\chi$, if $\pi$ 
is a generic symmetry
for $V$ or if $G$ is generically closed with respect to $V$, 
respectively.

As usual, the set of irreducible complex characters of $G$ 
is denoted by $\Irr(G)$.
We write $\rho_G$  
for the \emph{regular character} of $G$,
that is, the character of $\compl G$ as (left) 
module over itself.

Note that an arbitrary $\compl G$-module $V$ is cyclic 
if and only if $V$ is isomorphic to a left ideal of 
$\compl G$, 
because any epimorphism 
$\compl G \to V$ splits.
Thus a character $\chi$ is afforded by a cyclic 
$\compl G$-module if and only if 
$\chi$ is a constituent of $\rho_G$
(i.~e., $\rho_G-\chi$ is a character as well).
As $\rho_G = \sum_\psi \psi(1) \psi$, 
where $\psi$ runs over all irreducible characters of $G$,
an arbitrary character $\chi$ is afforded by a left ideal 
if and only if 
$\langle \chi, \psi \rangle \leq \psi(1)$
for all $\psi \in \Irr(G)$, 
where $\langle \; , \: \rangle$ denotes the usual 
inner product for class functions, i.~e. 
\[
\langle \alpha, \beta \rangle 
  = \frac{1}{|G|} \sum_{g \in G}
             \alpha(g) \overline{\beta(g)}\,.
\]
Unless otherwise stated, in the following every character 
is assumed to be afforded by a cyclic $\compl G$-module
(equivalently, a left ideal of $\compl G$).

We begin with the characterization of 
$\Sym(G,\chi)$ and $\widehat{\chi}$ for irreducible characters $\chi$, 
which is basically a reformulation of
Proposition~\ref{prp:abs_simp_gen_sym}.

\begin{co} \label{co:gen_sym_irred}
Let $\chi \in \Irr(G)$ and set $K:= \Ker(\chi)$. 
Then
\[
   \Sym(G,\chi) 
   = \{ \pi \in \Sym(G) \colon 
         \pi(gK) = \pi(1) g K
           \text{ for all } g \in G 
     \}\,.
\]
Furthermore, 
$\widehat{\chi}(\pi) = \chi(\pi(1))$ for all $\pi \in \Sym(G,\chi)$.
\end{co}
\begin{proof}
  Let $D\colon G\to \GL(V)$ be a representation affording $\chi$
  and suppose that $v\in V$ is generic, so that
  $K = \{g\in G \colon gv= v\}$.
  It follows easily from Lemma~\ref{lm:rep_kernel}
  that 
  \[D_v^{-1}(D(G)) 
    = \{ \pi \in \Sym(G)
         \colon 
         \pi(gK) = \pi(1)gK
         \text{ for all } g\in G
      \}\,,
  \]
  without any further assumption on $V$.
  But by Proposition~\ref{prp:abs_simp_gen_sym},
  we have $\Sym(G,\chi) = D_v^{-1}(D(G))$.
  Moreover, for $\pi\in \Sym(G,v)$, we have
  $D_v(\pi) = D(\pi(1))$ and thus
  $\widehat{\chi}(\pi) = \chi(\pi(1))$.
\end{proof}

In the proof of the next result,
we use the (unique) hermitian inner product $[\;,\:]$
on $\compl G$ such that $G$ is an orthonormal basis
with respect to $[\;,\:]$.
If $\pi \in \Sym(G)$, then $\pi$ extends uniquely to 
a linear automorphism of $\compl G$, also denoted by $\pi$,
which is clearly unitary with respect to this inner
product. 
In particular, left and right multiplications by elements
of $G$ are unitary.

For a subspace $L \leq \compl G$, 
we denote its orthogonal complement by
\[
   L^\perp  = 
     \{ x \in \compl G : [x,y] = 0 \text{ for all } y \in L 
     \}\,.
\]
It is easy to check that $L^\perp$ is a (left) ideal when $L$ is. 
Furthermore, for any $\pi \in \Sym(G)$ we have 
$\pi(L) \subseteq L$ if and only if 
$\pi(L^\perp) \subseteq L^\perp$. 

\begin{prp} \label{prp:gen_sym_leftideal}
A permutation $\pi \in \Sym(G)$ is a generic symmetry of 
the character $\chi$ 
if and only if 
one of the following equivalent conditions is satisfied:
\begin{enumthm}
\item \label{it:gen_sym_leftideal-1} 
      $\pi(L) \subseteq L$ for all left ideals $L$ affording $\chi$.
\item \label{it:gen_sym_leftideal-2} 
      $\pi(L) \subseteq L$ for all left ideals $L$ affording 
      $\rho_G - \chi$ (where $\rho_G$ is the regular character of $G$,
      as before).
\end{enumthm}
In particular, $\Sym(G,\chi) = \Sym(G,\rho_G - \chi)$.
\end{prp}
\begin{proof}
Let $V$ be a $\compl G$-module affording $\chi$,
and  
$v \in \Gen(V)$. 
By Lemma \ref{lm:gen_sym_by_ann}, 
$\pi$ is a generic symmetry of $\chi$ if and only if $\pi$ maps any 
isomorphic copy of $\Ann(v)$ in $\compl G$ onto itself. 
As $\compl G \cong V \oplus \Ann(v)$, 
these are precisely the left ideals affording
the character $\rho_G - \chi$, 
which shows \ref{it:gen_sym_leftideal-2}. 
The equivalence of \ref{it:gen_sym_leftideal-1} and 
\ref{it:gen_sym_leftideal-2} 
follows by taking orthogonal complements,
and by the fact that a left ideal $L$ is afforded by $\chi$ 
if and only if $L^\perp$ is afforded by $\rho_G - \chi$.
\end{proof}

Although we used properties of the complex numbers in the preceding
proof, the result of Proposition~\ref{prp:gen_sym_leftideal}
remains true for arbitrary fields of characteristic zero,
as explained at the beginning of this section.
On the other hand, if the characteristic of $\kk $ divides
the group order, then a left ideal of the group algebra $\kk G$
may not even be cyclic as $\kk G$-module.
(As an example, take the Klein four group $G=C_2\times C_2$
in characteristic~$2$. The kernel of $\kk G\to \kk $
is not cyclic as $\kk G$-module.)
And even when the characteristic does not divide the group order,
it is not true that a left ideal has the same generic symmetries 
as its complement. 
(An example exists with $G=C_7$ cyclic of order $7$
and $\kk $ of characteristic $2$.)

We continue to work over $\compl$, the field of complex numbers.
If a character $\chi$ is afforded by a twosided ideal 
$I$ of $\compl G$,
then $I$ is the unique left ideal of $\compl G$ affording $\chi$, 
and we call $\chi$ an \emph{ideal character}. 
Alternatively, a character $\chi$ is an ideal character
if and only if 
$\langle \chi , \psi \rangle \in \{ 0, \psi(1) \}$ for all 
$\psi \in \Irr(G)$.
In the following, we characterize $\Sym(G,\chi)$ 
and $\widehat{\chi}$ for
ideal characters $\chi$. 
The first statement of Proposition~\ref{prp:sym_of_ideal} is
essentially contained in our earlier 
paper \cite[Theorem 8.5]{FrieseLadisch16}, 
but we give a different proof here.

\begin{prp} \label{prp:sym_of_ideal}
Let $\chi$ be an ideal character of $G$. 
Then $\pi \in \Sym(G)$ is generic for $\chi$ if and only if
\[
  \chi(\pi(g)^{-1}\pi(h)) 
    = \chi(g^{-1}h) \quad\text{for all}\quad g,h \in G \,. 
\]
Furthermore, for all $\pi \in \Sym(G,\chi)$,
\[
\widehat{\chi}(\pi) 
   = \frac{1}{|G|} \sum_{g \in G} \chi(g^{-1} \pi(g))\,.
\]
\end{prp}
\begin{proof}
As $\chi$ is an ideal character, 
there is a twosided ideal $I \leq \compl G$ 
affording $\chi$, 
and $I$ is the unique left ideal affording $\chi$. 
The ideal $I$ is generated by the central 
idempotent~\cite[Theorem~2.12]{isaCTdov}
\[
   e = \frac{1}{\card{G}} \sum_{g \in G} \chi(g^{-1}) g \,.
\]

We claim that $\pi\in \Sym(G,\chi) =\Sym(G,I)$
if and only if $\pi(ge) = \pi(g) e$ for all $g\in G$.
The first assertion  of the proposition then follows
by comparing the coefficients 
of $\pi(h)^{-1}$ in 
the equation $\pi(ge) = \pi(g)e$.

The ``if'' direction of the claim is clear:
When $\pi(ge)=\pi(g)e$ for all $g\in G$, then
$\pi(I)\subseteq I$ and thus 
$\pi\in \Sym(G,\chi)$ by 
Proposition~\ref{prp:gen_sym_leftideal}.

For the ``only if'' direction, we 
first observe that $J= \compl G(1-e)$
is a twosided ideal and the unique left ideal
affording $\rho_G - \chi$.
When $\pi\in \Sym(G,\chi)$, then
$\pi(I)\subseteq I$ and $\pi(J)\subseteq J$,
by Proposition~\ref{prp:gen_sym_leftideal}.
Let $g\in G$ and consider the equation
\[
\pi(g) e = \pi(g e) e + \pi(g (1-e)) e \,.
\]
Since 
$\pi(g e) \in \compl G e$ and 
$\pi(g (1-e)) \in \compl G (1-e)$, 
it follows 
$\pi(g e) e = \pi(g e)$ and $\pi(g (1-e)) e = 0$.
Thus $\pi(g)e = \pi(ge)$, which shows the claim.

Now we prove the formula for $\widehat{\chi}$. 
The equation $\pi(ge) = \pi(g) e$ for all $g \in G$
shows that $D_e(\pi) = \pi_{|I} $,
that is, the restriction
$\pi_{|I}$ is the linear map that shows that
$\pi$ is an orbit symmetry for $e$.
The projection $e_r\colon \compl G\to I$
is given by (left or right) multiplication with $e$.
It follows \begin{align*}
  \widehat{\chi}(\pi) 
    & = \Tr_I(\pi) 
      = \Tr_{\compl G}(\pi \circ e_r) 
      = \frac{1}{\card{G}} 
         \sum_{g \in G} \rho_G( g^{-1} \pi(g e))
    \\
    & = \frac{1}{\card{G}} \sum_{g \in G} \rho_G( g^{-1} \pi(g) e) 
      = \frac{1}{\card{G}} \sum_{g \in G} \chi(g^{-1} \pi(g))\,.
      \qedhere
\end{align*}
\end{proof}

At this point we are able to recognize the generic symmetries of
ideal characters (Proposition~\ref{prp:sym_of_ideal}) and of
irreducible characters (Corollary~\ref{co:gen_sym_irred}).
These are in fact
all the necessary building blocks for recognizing generic symmetries
of arbitrary characters, as we will show now.

\begin{defn}
Let $\chi$ be a character. 
The \emph{ideal part} $\chi_I$ of $\chi$ is given by
\[
\chi_I = \sum_\psi \psi(1) \psi \,,
\]
where $\psi$ runs over all irreducible characters of $G$ with 
$\langle \chi, \psi \rangle = \psi(1)$. 
\end{defn}
If $L$ is any left ideal affording $\chi$, 
then $\chi_I$ is the character of the biggest twosided ideal 
contained in $L$.
In particular, $\chi_I$ is an ideal character.

\begin{thm} \label{thm:gen_sym_decomp}
A permutation $\pi \in \Sym(G)$ is generic for 
a character $\chi$ 
if and only if it is generic for $\chi_I$ 
and for any irreducible constituent of $\chi - \chi_I$.
\end{thm}
\begin{proof}
The ``if'' part is a direct consequence of Lemma~\ref{lm:gen_sym_sum}.

For the ``only if'' part, let $\pi$ be a generic symmetry of $\chi$. 
By Proposition~\ref{prp:gen_sym_leftideal}, 
we have $\pi(L) \subseteq L$ for all left ideals $L$ affording $\chi$. 
In particular, we have $\pi(I) \subseteq I$, 
where $I$ is the intersection of all these left ideals. 
$I$ is the biggest twosided ideal contained in any left ideal 
$L$ affording $\chi$, i.~e. $I$ is the ideal affording $\chi_I$. 
Hence, by Proposition~\ref{prp:gen_sym_leftideal}, 
$\pi$ is generic for $\chi_I$. 

Now let $\psi$ be any irreducible constituent of $\chi - \chi_I$, 
and let $S$ be any left ideal affording $\psi$. 
Then $S$ is contained in a left ideal $L$ affording $\chi$. 
Since $\psi$ is not a constituent of $\chi_I$, 
there is an isomorphic copy $S'$ of $S$ in $\compl G$ with 
$S' \cap L = 0$. 
Let $C$ be any complement of $L \oplus S'$ in $\compl G$. 
Then $S \oplus C \cong S' \oplus C$ affords $\rho - \chi$, 
whence $\pi(S \oplus C) \subseteq S \oplus C$ 
by Proposition~\ref{prp:gen_sym_leftideal}.
Finally, as $S = L \cap (S \oplus C)$, 
we conclude $\pi(S) \subseteq S$.
Since $S$ was arbitrary, 
$\pi$ is a generic symmetry of $\psi$ by 
Proposition~\ref{prp:gen_sym_leftideal}.
\end{proof}

Putting the previous results together, 
we get a characterization of $\Sym(G,\chi)$
in terms of $\chi$
(Theorem~\ref{ithm:gen_syms_char} from the introduction).

\begin{thm} \label{thm:gen_cls_by_char}
Let $\chi$ be a character of some cyclic $\compl G$-module. 
$\Sym(G,\chi)$ consists precisely of the permutations 
$\pi \in \Sym(G)$ satisfying the following conditions:
\begin{enumthm}
\item \label{it:idealpart_sym} 
      For all $g,h \in G$ we have
      \[ \chi_I( \pi(h)^{-1} \pi(g) ) 
         = \chi_I( h^{-1} g )\,.
      \]
\item For all $g \in G$ we have 
\[ \pi\big(g\Ker(\chi-\chi_I)\big) = \pi(1) g \Ker(\chi-\chi_I)\,. 
\]
\end{enumthm}
\end{thm}
\begin{proof}
By Theorem~\ref{thm:gen_sym_decomp}, 
a permutation $\pi \in \Sym(G)$ is a generic symmetry of $\chi$ 
if and only if it is a generic symmetry of $\chi_I$ and of any 
irreducible constituent of $\chi - \chi_I$.
By Proposition~\ref{prp:sym_of_ideal},
$\pi \in \Sym(G, \chi_I)$ is equivalent to \ref{it:idealpart_sym}.
By Corollary~\ref{co:gen_sym_irred},
$\pi \in \Sym(G,\psi)$ for $\psi\in \Irr(G)$
is equivalent to
$\pi(g\Ker(\psi)) =  \pi(1) g \Ker(\psi)$ for all $g \in G$.
Since $\Ker(\chi-\chi_I)$ is the intersection
of the kernels of its irreducible constituents,
the result follows.
\end{proof}

\begin{co}\label{co:gen_closed_suff_cond}
If $\chi-\chi_I$ is faithful, 
then $G$ is generically closed with respect to $\chi$.
\end{co}

\begin{prp} \label{prp:gen_character}
Let $\chi$ be a character of $G$. 
Then for all 
$\pi \in\Sym(G,\chi)$ we have
\begin{equation*}
\widehat{\chi}(\pi) = \frac{1}{|G|} \sum_{g \in G} \chi(g^{-1}
\pi(g))\,.
\end{equation*}
\end{prp}
\begin{proof}
By Theorem~\ref{thm:gen_sym_decomp}, 
it suffices to prove the claim for
ideal characters and for irreducible characters.

For ideal characters, the assertion is true by 
Proposition~\ref{prp:sym_of_ideal}. 
Let $\chi$ be irreducible. 
Then, by Corollary~\ref{co:gen_sym_irred}, 
$\pi(g)\in \pi(1) g \Ker(\chi)$ for all $g \in G$,
and thus $g^{-1}\pi(g) \in \pi(1) \Ker(\chi)$.
Therefore
\[
  \frac{1}{\card{G}} 
    \sum_{g \in G} \chi( g^{-1} \pi(g) ) 
   = \chi(\pi(1)) 
   = \widehat{\chi}(\pi)\,,
\]
where the last equation follows from Corollary~\ref{co:gen_sym_irred}.
\end{proof}

\section{Classification of affine symmetry groups 
\texorpdfstring{\\}{} of orbit polytopes}
\label{sec:answerbabai}
In this section, we classify the groups
which are isomorphic to the affine symmetry group
of an orbit polytope.
Suppose the finite group $G$ acts on $\reals^d$ 
by affine transformations. 
Recall that an \emph{orbit polytope} of $G$ is the 
convex hull of a $G$-orbit of a point~$v$:
\[  
   P(G,v) := \conv\{ gv \colon g\in G\}\,. 
\]

(As $G$ fixes the barycenter of $P(G,v)$,
 we can choose coordinates such that $G$ acts 
 linearly.)
Let us say that $P(G,v)$ is an \emph{euclidean} orbit polytope,
if $G$ acts by (euclidean) isometries on $\reals^d$.
The euclidean symmetry group (or isometry group)
of a $d$-dimensional polytope 
$P\subseteq \reals^d$ 
is the group of isometries of $\reals^d$ mapping $P$ onto itself.
(In the literature, the euclidean symmetry group
 is often called 
 ``the'' symmetry group of $P$.
 For the sake of clarity, we do not follow this convention 
 here.)
The affine symmetry group of a $d$-dimensional polytope
$P \subseteq \reals^d$ is the group of all
affine transformations of $\reals^d$ mapping $P$ onto itself.
In both cases, a symmetry maps $P$ onto itself
if and only if it
permutes the vertices of~$P$.

Babai~\cite{Babai77} classified the finite groups 
which are isomorphic to the isometry group
of an euclidean orbit polytope, and
asked which abstract finite groups occur as
the affine symmetry group of an orbit polytope.
In this section, we answer this question.
We begin by recalling Babai's classification.
Following Babai, we call a finite group
$G$ \emph{generalized dicyclic}, if it has an abelian
subgroup $A$ of index $2$ and an element $g\in G\setminus A$
of order $4$ such that $g^{-1}ag=a^{-1}$ for all $a\in A$.

\begin{thm}[Babai~\cite{Babai77}] \label{thm:Babai77}
  Let $G$ be a finite group.
  Then $G$ is not isomorphic to the isometry group 
  of an euclidean orbit polytope, 
  if and only if one of the following holds:
  \begin{enumthm}
  \item $G$ is abelian, but not elementary $2$-abelian.
  \item $G$ is generalized dicyclic.
  \end{enumthm}
  Any other finite group is isomorphic to the isometry group 
  of an euclidean orbit polytope.
\end{thm}

Now if a finite group $G$ (say) is the affine symmetry group 
of a polytope $P\subseteq \reals^d$,
then there is an affine automorphism 
$\sigma$ of $\reals^d$ 
such that $\sigma G \sigma^{-1}$ preserves lengths.
Since  $\sigma G \sigma^{-1}$ is the affine 
symmetry group of the polytope $\sigma(P)$, it is also the 
euclidean symmetry group of the polytope $\sigma(P)$.
Thus as an immediate corollary of Babai's result,
we get the following:
\begin{co}\label{co:babaicor}
  The following groups are not isomorphic to the affine symmetry group
  of an orbit polytope: abelian groups of exponent greater than $2$,
  and generalized dicyclic groups.
\end{co}
On the other hand, it may happen that $G$ is isomorphic 
to the isometry group of an (euclidian)
orbit polytope, but not to the affine symmetry group of an
orbit polytope.
For example, the Klein four group, $V_4$,
is isomorphic to the isometry group of a rectangle
with two different side lengths.
The affine symmetry group of a rectangle
is isomorphic to the group of the square and has order $8$,
and indeed, the Klein four group can not be realized as
the affine symmetry group of an orbit 
polytope~\cite[Lemma~9.1]{FrieseLadisch16}.

In order to apply the results of the previous sections
to our classification problem,
we need the following observation.
\begin{lemma} \label{lm:affsym_genclosed}
  A finite group $G$ is isomorphic to the affine symmetry group of
  an orbit polytope
  if and only if $G$ is generically closed with respect to some 
  cyclic module over\/ $\reals G$.
\end{lemma}
\begin{proof} 
  If $G$ is generically closed with respect to some
  cyclic module $V $,
  then $G$ is isomorphic to the linear symmetry group of
  the orbit polytope $P(G,v)$ for every generic $v\in V$.
  It is not difficult to see
  that the barycenter $(1/\card{G}) \sum_{g\in G} gv$
  is the only point 
  in the \emph{affine} hull of $Gv$ which is fixed
  by $G$~\cite[Lemma~2.1]{FrieseLadisch16}.
  Thus either $0 = (1/\card{G})\sum_g gv$ and the affine
  and the linear symmetry group coincide, or
  $0\neq (1/\card{G})\sum_g gv$ and the affine and the linear 
  symmetry group of $P(G,v)$ are isomorphic
  (by restriction from the linear to the affine hull
  of $Gv$).
  
  Conversely, suppose that $G$ is the affine symmetry group of 
  an orbit polytope $P(H,v)$ of a group $H$.
  Then clearly $P(H,v) = P(G,v)$.     
  Without loss of generality, we may assume that
  $H$ and $G$ are linear, by choosing the barycenter
  of $P(G,v)$ as origin of our coordinate system.
  Set $V = \reals G v$, the $\reals$-linear span of $Gv$, 
  so that $V$ is a cyclic $\reals G$-module.
  
  We have $Hv=Gv$ and
  $G = \GL(Hv)$.
  Corollary~\ref{co:closure} yields
  that $G = \GL(Gw)$ for any $w\in V$ 
  which is generic for $G$.
  Since $G$ is by definition a subgroup of $\GL(V)$,
  this is equivalent to $G$ being generically closed
  with respect to $V$.
\end{proof}

We need to recall some representation theory.
We have already seen that a character 
\[ \gamma = \sum_{\chi\in \Irr(G)} n_{\chi}\chi 
\]
is afforded by a cyclic $\compl G$-module,
or a left ideal of $\compl G$, 
if and only if $n_{\chi}\leq \chi(1)$ for all
$\chi\in \Irr(G)$.
We now characterize which characters are afforded 
by a left ideal of $\kk G$, where
$\kk \subseteq \compl$.

Let $\chi\in \Irr G$.
Recall that the \emph{Schur index} 
$m_{\kk }(\chi)$ of $\chi$ over $\kk $ 
is by definition the smallest positive integer
$m$ such that $m\chi$ is afforded by a representation with
entries in $\kk (\chi)$,
where $\kk (\chi)$ is the field generated by the values of $\chi$.

For $\kk = \reals$,
only three different cases are possible,
which can be recognized
by the \emph{Frobenius-Schur-indicator}
\[ \FS(\chi) := \frac{1}{\card{G}}
          \sum_{g\in G} \chi(g^2)\,.
\]
Namely, when
$\FS(\chi) = 1$, then
$\chi = \overline{\chi}$ and $m_{\reals}(\chi)=1$,
so $\chi$ is afforded by a representation over $\reals$.
When $\FS(\chi) = 0$, then
$\chi\neq \overline{\chi}$ and $m_{\reals}(\chi)=1$.
Finally, when $\FS(\chi) = -1$, then
$\chi = \overline{\chi}$, but $\chi$ is not afforded by a 
representation over $\reals$,
and $m_{\reals}(\chi) = 2$ \cite[Chapter~4]{isaCTdov}.

\begin{lemma}\label{lm:char_kmod}
  Let\/ $\kk \subseteq \compl$, and let
  \[ \gamma = \sum_{\chi\in \Irr(G)} n_{\chi}\chi
  \]
  be a character. 
  Then $\gamma$ is the character of a left ideal of\/ $\kk G$
  if and only if the following conditions hold:
  \begin{enumthm}
  \item $\gamma$ has values in $\kk $,
  \item $m_{\kk }(\chi)$ divides $n_{\chi}$ 
        for all $\chi\in \Irr G$,
  \item $n_{\chi} \leq \chi(1)$ for all $\chi \in \Irr(G)$.
  \end{enumthm}
\end{lemma}
\begin{proof}
    It follows from the general theory of the Schur index
    that $\gamma$ is the character of a representation 
    with entries in $\kk $
    if and only if the first two conditions 
    hold \cite[Corollary~10.2]{isaCTdov}. 
    
    Let $S$ be a simple $\kk G$-module.
    The character of $S$ has the form
    \[ m_{\kk }(\chi) \sum_{\alpha} \chi^{\alpha}\,,
    \]
    for some $\chi\in \Irr(G)$,
    where $\alpha$ runs over the Galois group
    of $\kk (\chi)/\kk $, so that
    $\chi^{\alpha}$ runs over the
    Galois conjugacy class of $\chi$ over $\kk $.
    Thus $S$ occurs with multiplicity 
    $\chi(1)/m_{\kk }(\chi)$ as summand of the 
    regular module $\kk G$, 
    and with multiplicity $n_{\chi}/m_{\kk }(\chi)$
    in a $\kk G$-module affording $\gamma$.
    Thus a $\kk G$-module affording $\gamma$ is a direct summand 
    of $\kk G$ if and only if $n_{\chi}\leq \chi(1)$
    for all $\chi$.
\end{proof}

\begin{defn}
For a finite group $G$ and a field $\kk \subseteq \compl$, set 
\[ \NKer_{\kk }(G) :=
   \bigcap \{ \Ker(\chi)\colon \chi(1) > m_{\kk }(\chi)
           \}\,.
\]
If $\chi(1)= m_{\kk }(\chi)$ for all $\chi\in \Irr G$,
then we set $\NKer_{\kk }(G)=G$.
\end{defn}

Consider the character
\[ \gamma := \sum_{ \substack{ \chi\in \Irr G \\ 
                               \chi(1) > m_{\kk }(\chi) 
                             }
                  } m_{\kk }(\chi) \chi\,.
\]
Then $\Ker \gamma = \NKer_{\kk }(G)$, 
and $\gamma$ is afforded by a left ideal of $\kk G$,
and the ideal part of $\gamma$ is zero.
So as a corollary of 
Corollary~\ref{co:gen_closed_suff_cond} 
and Lemma~\ref{lm:affsym_genclosed},
we get the following.

\begin{co} \label{co:trivial_sker}
Let $\kk \subseteq \reals$.
If\/ $\NKer_{\kk }(G) = \{1\}$, then
$G$ is isomorphic to the affine symmetry group of 
an orbit polytope, such that its vertices
have coordinates in $\kk $.
\end{co}

In particular, when $\NKer_{\reals}(G)=\{1\}$,
then $G$ is isomorphic to the affine symmetry group of an 
orbit polytope.
It remains to treat the groups $G$ for which 
$\NKer_{\reals}(G) \neq \{1\}$.
The following theorem gives the complete list of groups $G$ with
non-trivial $\NKer_{\reals}(G)$.
Here, $Q_8$ denotes the quaternion group of order~$8$,
and $C_n$ a cyclic group of order~$n$. 
\begin{prp}[{\cite[Theorem~B]{ladisch16bpre}}]
\label{prp:trivial_sker}
$G$ is a finite group with $\NKer_{\reals}(G) \neq \{1\}$,
if and only if
one of the following holds:
\begin{enumthm}
\item $G$ is abelian, and $G\neq \{1\}$.
\item $G$ is generalized dicyclic.
\item $G$ is isomorphic to $Q_8 \times C_4 \times C_2^r$ for some $r
\geq 0$.
\item $G$ is isomorphic to $Q_8 \times Q_8 \times C_2^r$ for some $r
\geq 0$.
\end{enumthm}
\end{prp}

\begin{lemma} \label{lm:Q8C4C2_closed}
  Let $G = Q_8 \times C_4 \times C_2^r$ for some $r \geq 0$. 
  Then $G$
  is generically closed with respect to some 
  $\reals G$-module.
\end{lemma}
\begin{proof}
Let $\alpha\in \Irr Q_8$ be the faithful irreducible character 
of degree $2$. 
Let $\beta = \lambda + \con{\lambda}$,
where $\lambda$ is a faithful linear character of 
$C_4 = \langle c  \rangle$.
Finally, let $\gamma$ be a faithful ideal character of $C_2^r$.
Then we claim that 
\[ \chi := \alpha \times \beta \times \gamma 
          + 2 \alpha \times 1 \times 1 
          + 1 \times \beta \times 1
\] 
is afforded by a left ideal of $\reals G$, 
and that $G$ is generically closed with respect to $\chi$.

The irreducible constituents of the first summand have the form
$\tau =\alpha \times \lambda \times \sigma$,
where $\lambda\in \Lin(C_4)$ is faithful and 
$\sigma \in \Lin(C_2^r)$.
We have $\tau \neq \con{\tau}$ and 
$\tau(1)=2 > m_{\reals}(\tau) =1$,
and both $\tau$ and $\con{\tau}$ occur in $\chi$ with multiplicity~$1$.

The other irreducible constituents of $\chi$ occur with 
multiplicity $m_{\reals}(\chi)$.
Thus $\chi$ is afforded by a left ideal of $\reals G$,
and the ideal part of $\chi$ is
\[ \chi_I = 2 \alpha \times 1 \times 1 + 1 \times \beta \times 1\,.
\] 

An easy calculation shows 
$\Ker(\chi-\chi_I) = \Ker(\alpha \times \beta \times \gamma) 
  = \langle u \rangle$ 
with $u = (-1,c^2,1)$. 
We have $\chi_I(u) = -6 = - \chi_I(1)$, 
so $u$ is in the center of $\chi_I$. 
In particular, $\chi_I(gu) = -\chi_I(g)$ for all $g \in G$. 

Let $\pi \in \Sym(G,\chi)$ be any generic symmetry with 
$\pi(1)= 1$. 
By Theorem~\ref{thm:gen_cls_by_char}, 
$\pi$ leaves the left cosets of $\Ker(\chi - \chi_I)$ setwise fixed. 
So if $\pi$ is not the identity,
then there is an element $g \in G$ with $\pi(g) = gu$.
Then, again by Theorem~\ref{thm:gen_cls_by_char}, we have
$\chi_I(g) = \chi_I(\pi(g)) = -\chi_I(g)$, 
so $\chi_I(g) = 0$ 
which means $g = (x,y,z)$ with 
$x \in Q_8 \setminus \Zent(Q_8)$ and $y^2 \neq 1$. 
Let $h = (x,1,1)$. 
Then $0 \neq \chi_I(h) = 2$, 
so $\pi(h) = h$, and $0 \neq \chi_I(h^{-1}g) = 2$. 
But then
$\chi_I(\pi(h)^{-1} \pi(g)) = \chi_I(h^{-1} u g) = -\chi_I(h^{-1}g)$ 
is a contradiction to Theorem~\ref{thm:gen_cls_by_char},
which shows that $\pi$ must be the identity. 
Hence, $G$ is generically closed with respect to $\chi$.
\end{proof}

\begin{lemma} \label{lm:Q8Q8C2_closed}
Let $G = Q_8 \times Q_8 \times C_2^r$ for some $r \geq 0$. Then $G$
is generically closed with respect to some $\reals G$-module.
\end{lemma}
\begin{proof}
Let $\alpha \in \Irr(Q_8)$ be as in the proof of
Lemma~\ref{lm:Q8C4C2_closed},
and $\gamma$ a faithful character of $C_2^r$,
also as above. 
We claim that the character
\[ \chi = \alpha \times \alpha \times \gamma 
          + 2 \alpha \times 1 \times 1 
          + 1 \times 2\alpha \times 1
\] 
is afforded by a left ideal of $\reals G$, 
and that $G$ is generically closed with respect to $\chi$.

The proof follows the same lines as in Lemma~\ref{lm:Q8C4C2_closed}.
For the same reasons as above, 
$\chi$ is afforded by a left ideal of $\reals G$. 
Its ideal component is given by 
\[ \chi_I = 2 \alpha \times 1 \times 1 
           + 1 \times 2 \alpha \times 1\,,
\] 
and we have
$\Ker(\chi - \chi_I) = \Ker(\alpha \times \alpha \times \gamma) 
                     = \langle u \rangle$
with $u = (-1,-1,1)$. 
We have $\chi_I(u)
= -8 = -\chi_I(1)$, so $u$ is in the center of $\chi_I$. 
Let $\pi\in \Sym(G,\chi)$ with $\pi(1) = 1$. 
If $\pi$ is not the identity
then there is an element $g \in G$ with $\pi(g) = g u$, 
so again by Theorem~\ref{thm:gen_cls_by_char}, 
$\chi_I(g) = 0$. 
Therefore, $g = (x,y,z)$ with 
$(x,y) \in \{ (1,-1), (-1,1) \}$ or 
$x$, $y \in Q_8 \setminus \Zent(Q_8)$. 
In the first case, 
set $h = (a,1,1)$, where $a$ is any element in 
$Q_8 \setminus \Zent(Q_8)$; 
in the second case, set $h = (x,1,1)$. 
In both cases, we have $\chi_I(h) \neq 0$, so $\pi(h) = h$, 
and $\chi_I(h^{-1}g) \neq 0$. 
Again, as in the proof above, we obtain a
contradiction to Theorem~\ref{thm:gen_cls_by_char} by
$\chi_I(\pi(h)^{-1} \pi(g)) = \chi_I(h^{-1} u g) = -
\chi_I(h^{-1}g)$. 
So $\pi$ must be the identity, and $G$ is
generically closed with respect to $\chi$.
\end{proof}

The following is Theorem~\ref{ithm:class_aff} 
from the introduction:

\begin{thm} \label{thm:class_aff_sym_grps}
Let $G$ be a finite group.
Then $G$ is isomorphic to the affine symmetry group
of an orbit polytope, 
if and only if none of the following holds:
\begin{enumthm}
\item $G$ is abelian of exponent greater than $2$.
\item $G$ is generalized dicyclic.
\item $G$ is elementary abelian of order $4, 8$ or $16$.
\end{enumthm}
\end{thm}
\begin{proof}
It follows from Corollary~\ref{co:babaicor} that groups
which are abelian, but not elementary $2$-abelian, 
are not isomorphic to the affine symmetry group of an orbit polytope,
and the same holds for generalized dicyclic groups.
The case of elementary abelian $2$-groups was dealt with in
our previous paper~\cite[Theorem~9.9]{FrieseLadisch16}:
An elementary abelian $2$-group $G$ is isomorphic to the affine 
symmetry group of an orbit polytope if and only if its order,
$\card{G}$, is not $4$, $8$ or $16$.
The groups 
$Q_8 \times C_4 \times C_2^r$ and 
$Q_8 \times Q_8 \times C_2^r$ 
are affine symmetry groups of orbit polytopes for any 
$r \geq 0$ 
by Lemma~$\ref{lm:Q8C4C2_closed}$ 
and Lemma~\ref{lm:Q8Q8C2_closed}. 
All remaining groups are affine symmetry
groups of orbit polytopes by 
Proposition~\ref{prp:trivial_sker} and
Corollary~\ref{co:trivial_sker}.
\end{proof}

We end this section with a related question.
It is not difficult to show that when $G$ is a nonabelian
group, then the intersection of the kernels of
all \emph{nonlinear} irreducible characters is the trivial 
subgroup \cite[Lemma~3.1]{ladisch16bpre}.
Thus when $V$ is a $\compl G$-module affording the character
$ \gamma = \sum_{\chi} \chi$, 
where the sum runs over the nonlinear irreducible 
characters of $G$,
then $\Sym(G,V)\iso G$.

\begin{ques}
  Which finite abelian groups $G$ are generically closed
  with respect to some representation
  $G\to \GL(d,\compl)$?
  Equivalently, which finite abelian groups are isomorphic
  to the linear symmetry group of some 
  point set in $\compl^d$ for some $d$,
  and act transitively on this set?
\end{ques}

We conjecture that there are only finitely many abelian groups
(up to isomorphism)
that are not generically closed with respect to at least one
representation. 
By Proposition~\ref{prp:abs_simp_gen_sym},
every cyclic group is generically closed with respect
to some faithful linear representation.
Our earlier result~\cite[Theorem~9.9]{FrieseLadisch16},
together with Proposition~\ref{prp:scalar_ext},
answers the above question for elementary abelian $2$-groups.
In particular, the elementary abelian
$2$-groups of orders $4$, $8$ and $16$ are not generically closed
with respect to some representation.
One can check that the elementary abelian $3$-group
$C_3\times C_3$
of order~$9$ is also not generically closed with respect
to any representation. We conjecture that these four groups
are the only such groups.

\section{Classification of affine symmetry groups
        \texorpdfstring{\\}{}
         of rational orbit polytopes}
\label{sec:class_q}
In this section, we classify affine symmetry groups
of polytopes with rational coordinates.
Since every polytope with rational coordinates can be scaled
to a polytope with integer coordinates, this classifies also
affine symmetry groups of lattice orbit polytopes.

By Corollary~\ref{co:trivial_sker},
it follows that when $\NKer_{\QQ}(G) = \{1\}$, then
$G$ is isomorphic to the affine symmetry group of 
an orbit polytope with integer coordinates.
The main result of this section depends on the
classification of the finite groups $G$
with $\NKer_{\QQ}(G)\neq \{1\}$ 
\cite[Theorem~D]{ladisch16bpre}.

The next lemma can often be used to show
 that a certain group is
\emph{not} the affine symmetry group of a rational orbit polytope.
As a consequence of the classification
of the finite groups $G$
with $\NKer_{\QQ}(G)\neq 1$ \cite[Theorem~D]{ladisch16bpre}, 
it turns out that most of these groups satisfy 
the assumptions of the next lemma.

\begin{lemma}\label{lm:concr_aut_gensym}
  Suppose $G$ has a normal subgroup $N$ of prime index
  $\card{G:N}= p$ and an element $z$ of order $p$,
  such that $z\in \erz{g}$ for every $g\in G\setminus N$.
  Fix an epimorphism 
  $\kappa\colon G\to \erz{z}$ with kernel $N$
  and define $\alpha\colon G\to G$ by
  $\alpha(g) = g \kappa(g)$.
  Then $\alpha\in \Sym(G,I)$ for every
  ideal $I$ of $\QQ G$.
  If additionally $z\in \NKer_{\QQ}(G)$,
  then $G$ can not be isomorphic to the affine symmetry group
  of an orbit polytope with rational coordinates.
\end{lemma}
\begin{proof}
  First notice that $z\in \Zent(G)$, 
  since $G\setminus N$ centralizes $z$ by assumption.
  This yields that $\alpha$ is a group automorphism
  of $G$, with inverse $g\mapsto g \kappa(g)^{-1}$.

  By Lemma~\ref{lm:gen_sym_sum}, it suffices to assume
  that $I$ is a simple ideal.
  Then the character of $I$
  has the form
  $\gamma = \chi(1) \sum_{\sigma} \chi^{\sigma}$
  for some $\chi\in \Irr(G)$,
  where $\sigma$ runs over the Galois group
  of $\QQ(\chi)/\QQ$.
  If $z\in  \Ker(\chi)$, then $z\in \Ker(\gamma)$ 
  and the result is clear
  by the remarks before Lemma~\ref{lm:gen_sym_sum},
  or by Proposition~\ref{prp:sym_of_ideal}.
  
  So assume that $z\notin \Ker(\chi)$.
  Since $z\in \Zent(G)$,
  we have that $\chi(z) = \chi(1)\zeta$ for some
  primitive $p$-th root of unity,~$\zeta$.
  Let $g\in G\setminus N$ be arbitrary.
  The restriction $\gamma_{|\erz{g}}$ 
  decomposes into a sum of Galois orbits of linear 
  characters.
  Since $\chi(z) = \chi(1)\zeta $,
  we have $\lambda(z) = \zeta\neq 1$ for each 
  linear constituent $\lambda$ of $\gamma$.
  Since $\card{G/N} = \card{\erz{z}} = p$ 
  and $z\in \erz{g}$, we see that 
  $\lambda(g)$ is a primitive $k$-th root of unity where
  $p^2$ divides $k$.

  The Galois orbit of $\lambda(g)$ consists of all
  primitive $k$-th roots of unity.
  Since $p^2$ divides $k$, 
  the Galois orbit of $\lambda(g)$ is a union of
  cosets of $\erz{\zeta}$, and so the sum over the Galois orbit 
  is zero.
  It follows that $\gamma(g) = 0$.
  As $g\in G\setminus N$ was arbitrary, it follows that
  $\gamma_{G\setminus N} \equiv 0$.
  Since $\alpha$ is a group automorphism 
  leaving each element of $N$ fixed,
  we have that
  \[ 
     \gamma(\alpha(g)^{-1}\alpha(h)) 
      = \gamma( \alpha(g^{-1}h)) 
      = \gamma( g^{-1}h )      
  \]
  for all $g$, $ h\in G$.
  By Proposition~\ref{prp:sym_of_ideal}, this shows
  that $\alpha\in \Sym(G,\gamma)=\Sym(G,I)$.
  
  If $z\in \NKer_{\QQ}(G)$, then 
  $z\in \Ker(\gamma -\gamma_{I})$ for every character $\gamma$
  of a cyclic $\QQ G$-module,
  where $\gamma_I$ denotes the ideal part of $\gamma$,
  as before.
  It follows from Theorem~\ref{thm:gen_cls_by_char}
  that $\alpha \in \Sym(G,\gamma)$ for such $\gamma$.
  Since $\alpha(1_G) = 1_G$, but $\alpha \neq \id_G$,
  we have $\card{\Sym(G,\gamma)} > \card{ G }$.
\end{proof}

\begin{thm}\label{thm:class_int_polyt}
  The finite group $G$ is the affine symmetry of 
  an orbit polytope with vertices 
  with integer (rational) coordinates
  if and only if none of the following holds:
    \begin{enumthm}
    \item $G$ is abelian and either 
          $G$ has exponent greater than $2$,
          or $G$ is elementary abelian
          of order $4$, $8$ or $16$.
    \item \label{it:m_nilp}
          $G = S \times A$,
          where $S$ is a generalized dicyclic group of exponent $4$,
          the group $A$ is abelian of odd order,
          and the multiplicative order of\/
          $2$ modulo $\card{A}$ is odd.
    \item \label{it:m_gendic}
          $G$ is generalized dicyclic.
    \item \label{it:m_allab}
          $G = (PQ) \times B$,
          where the subgroups
          $P\in \Syl_p(G)$, $Q\in \Syl_q(G)$ and
          $B$ are abelian, 
          $P = \erz{g, \Cent_P(Q)} $
          and  there is
          some integer $k$ such that, 
          $x^g = x^k$ for all $x\in Q$.
          If $p^c = \card{P/\Cent_P(Q)}$,
          then $p^d = \ord( g^{p^c} )$ is the exponent 
          of $\Cent_P(Q)$,
          and $(q-1)_p$, the $p$-part of $q-1$,
            divides $p^d$.
          Finally, the $p$-part of the multiplicative order
          of $q$ modulo $\card{B}$ divides the multiplicative
          order of $q$ modulo $p^d$.
    \item \label{it:m_dirprod_h}
          $G = Q_8 \times (C_2)^r \times H$,
          where $H$ is as in \ref{it:m_allab} and has odd order,
          and the multiplicative order of\/
          $2$ modulo $\card{H}$ is odd.
    \end{enumthm}  
\end{thm}
\begin{proof}
  When $G$ is not the affine symmetry group of an orbit polytope
  with lattice points as vertices, then
  $\NKer_{\QQ}(G)\neq 1$.
  The list of such groups consists of the groups in 
  the above list, and the following groups:
  \begin{enumerate}[label= (\alph*)]
  \item $G = Q_8 \times C_4 \times (C_2)^r \times A$,
  \item $G = Q_8 \times Q_8 \times (C_2)^r \times A$,
  \end{enumerate}
  where in each case $A$ is abelian of odd order,
  and the multiplicative order of~$2$ modulo $\card{A}$ is 
  odd~\cite[Theorem~D]{ladisch16bpre}.
  However, these groups can be realized as symmetry groups 
  of integer orbit polytopes. 
  This can be shown as in Lemmas~\ref{lm:Q8C4C2_closed}
  and~\ref{lm:Q8Q8C2_closed}.
  (Replace the character $\gamma$ in these proofs by a 
  faithful ideal character of $(C_2)^r\times A$
  with values in $\QQ$.)
  
  If $G$ is abelian or generalized dicyclic,
  but not an elementary abelian $2$-group, 
  then $G$ is not even the affine symmetry group
  of an orbit polytope.
  For elementary abelian $2$-groups
  of order not $4$, $8$, or $16$,
  we constructed in fact 
  orbit polytopes with 
  integer coordinates~\cite[Theorem~9.9]{FrieseLadisch16}.
  For all other groups on the above list,
  Lemma~\ref{lm:concr_aut_gensym} applies. 
  For example, when $G=(PQ)\times B$ as in 
  \ref{it:m_allab}, then we choose for $N$ the unique 
  subgroup containing $\Cent_P(Q)QB$ of index~$p$,
  and $z = g^{p^{c+d-1}}$.
  Notice that $p^c$ divides $(q-1)_p$
  and $(q-1)_p$ divides $p^d$,
  so $z\in \Zent(G)$.
  For $u\in P\cap N$, we have
  $gu=ug$ and $\ord(u) < p^{c+d}$ by assumption,
  so $(gu)^{p^{c+d-1}}=z$.
  For $x\in Q$, we have $(gx)^{p^c}=g^{p^c}$.
  Since $g$ centralizes $B$, we have
  $z\in \erz{gn}$ for all $n\in N$,
  and also $z\in \erz{h}$ for all
  $h\in G\setminus N$.
  Moreover, $z\in \NKer_{\QQ}(G)$ 
  \cite[Lemma~6.11]{ladisch16bpre} and thus 
  Lemma~\ref{lm:concr_aut_gensym} applies.
  The same argument applies to the groups
  in~\ref{it:m_dirprod_h}
  (with $N$ containing $Q_8\times(C_2)^r$).
  If $G$ is as in ~\ref{it:m_nilp},
  then we choose for $N$ the direct product of $A$
  and the abelian subgroup of $S$ from the definition
  of ``generalized dicyclic''.
  (In fact, Lemma~\ref{lm:concr_aut_gensym} applies also
  to generalized dicyclic groups.)
\end{proof}

Thus there are quite a number of groups which can be realized
as symmetry groups of orbit polytopes, but not as symmetry group
of an orbit polytope with rational or integer coordinates.
As an example, consider the group
$G= Q_8 \times C_7$. 
Then it is known that $\QQ G$ is a direct product 
of division rings~\cite{ladisch16bpre,Sehgal75}.
This is essentially due to the fact that 
$\QQ(\eps)$, where $\eps$ is a primitive $7$-th root of unity,
is not a splitting field of the quaternions over $\QQ$.
Equivalently,
$-1$ is not a sum of two squares in $\QQ(\eps)$.
More generally, for a field $\kk $, we have that
$\kk G$ is a direct product of division rings
if and only if $-1$ is not a sum of two
squares in $\kk (\eps)$ \cite[Theorem~4.2]{ladisch16bpre}.
When $\kk G$ is not a direct product of division rings,
then $\kk G$ contains a simple left ideal which is not a twosided ideal
and on which $G$ acts faithfully.
Thus when $\kk \subseteq \reals $,
then $G$ is isomorphic to the affine symmetry group of an orbit 
polytope with vertex coordinates in $\kk $
if and only if 
$-1$ is a sum of two squares in $\kk (\eps)$.
There are many different such fields, 
for example, $\kk =\QQ(\sqrt{2})$ or $\kk =\QQ(\sqrt{5})$,
and also the following fields:
Choose $\alpha$, $\beta\in \reals$
with $\alpha^2 + \beta^2 = 7$
(note that $\alpha$ can be transcendent).
Then $\QQ(\alpha,\beta,\eps)$ is a splitting field
for the quaternions,
because $-7$ is a square in $\QQ(\eps)$.
Thus $G$ is isomorphic to the symmetry group of an orbit
polytope with coordinates in $\kk =\QQ(\alpha, \beta)$.
When $\alpha$ is transcendent, then one can show that
$\QQ(\alpha,\beta)$ contains no algebraic elements.

%
\printbibliography
%
%
%

\end{document}